\documentclass[dvipsnames]{article}

\usepackage[margin = 1.4in]{geometry}

\usepackage{tikz}
\usepackage{tikz-cd}
\usepackage{float}
\usepackage[all]{xy} 
\usepackage{amsmath}

\tikzset{->-/.style={decoration={
  markings,
  mark=at position .45 with {\arrow{>}}},postaction={decorate}}}
\usepackage{amsmath}
\usepackage{amssymb}
\usepackage{amsthm}
\usepackage{esvect}
\usepackage[shortlabels]{enumitem}
\usepackage{graphicx}
\usepackage{mathdots}
\usepackage{mathtools}
\usepackage{diagbox}
\usepackage{array, makecell}
\usepackage{rotating}
\usepackage{adjustbox}
\usetikzlibrary{arrows}
\usepackage{aliascnt}
\usetikzlibrary{decorations.pathmorphing}

\usepackage{xcolor}
\usepackage[hidelinks]{hyperref} 
\hypersetup{
    colorlinks=true,
    linkcolor=MidnightBlue,
    filecolor=MidnightBlue,      
    urlcolor=MidnightBlue,
    citecolor=MidnightBlue,
    pdfpagemode=FullScreen,
    }

\numberwithin{equation}{section}

\newcommand{\mynewtheorem}[2]{
  \newaliascnt{#1}{dummy}
  \newtheorem{#1}[#1]{#2}
  \aliascntresetthe{#1}
  \expandafter\def\csname #1autorefname\endcsname{#2}
}

\theoremstyle{plain}
\mynewtheorem{theorem}{Theorem}
\mynewtheorem{lemma}{Lemma}
\mynewtheorem{proposition}{Proposition}
\mynewtheorem{corollary}{Corollary}
\mynewtheorem{conjecture}{Conjecture}

\theoremstyle{definition}
\mynewtheorem{example}{Example} 
\mynewtheorem{speculation}{Speculation} 
\mynewtheorem{definition}{Definition}
\mynewtheorem{remark}{Remark}
\mynewtheorem{notation}{Notation}
\mynewtheorem{question}{Question}

\def\C{\mathbb{C}}

\def\Z{\mathbb{Z}}

\def\N{\mathbb{N}}
\def\R{\mathbb{R}}

\def\P{\mathbb{P}}
\def\M{\mathcal{M}}
\def\cM{\overline{\mathcal{M}}}

\def\cJ{\mathcal{J}}
\def\cH{\mathcal{H}}
\def\G{\mathsf{\Gamma}}
\def\cG{\widetilde{\mathsf{\Gamma}}}
\newcommand{\Tev}{{\mathsf{Tev}}}
\newcommand{\vTev}{{\mathsf{vTev}}}
\DeclareMathOperator{\im}{\mathrm{im}}
\DeclareMathOperator{\delbar}{\overline\partial}
\newcommand{\cB}{\mathcal{B}}
\newcommand{\cE}{\mathcal{E}}
\newcommand{\cS}{\mathcal{S}}

\newcommand{\cF}{\mathcal{F}}
\newcommand{\bg}{\underline{\mathsf{g}}}
\newcommand{\bp}{\underline{\mathsf{p}}}
\newcommand{\bq}{\underline{\mathsf{q}}}
\newcommand{\bd}{\underline{\mathsf{d}}}
\newcommand{\bz}{\underline{\mathsf{z}}}
\newcommand{\bx}{\underline{\mathsf{x}}}
\newcommand{\bA}{\underline{\mathsf{A}}}
\newcommand{\bm}{\underline{\mathsf{m}}}
\newcommand{\bv}{\underline{\mathsf{v}}}
\newcommand{\ev}{\mathrm{ev}}
\DeclareMathOperator{\coker}{coker}

\title{Pseudo-holomorphic curves with a fixed complex structure in positive symplectic manifolds}
\author{Alessio Cela and Aleksander Doan\\[0.5cm]}
\date{\vspace{-5ex}}

\begin{document}

\maketitle
\begin{abstract}    \noindent We prove a symplectic version of a conjecture of Lian and Pandharipande: in sufficiently high degree, the fixed-domain Gromov--Witten invariants of positive symplectic manifolds are signed counts of pseudo-holomorphic curves. The original conjecture in the complex algebraic setting was recently disproved by Beheshti et al. However, we show that the statement holds when the complex structure is replaced by a generic almost complex structure. The proof relies on showing that the fixed-domain Gromov–Witten pseudocycle can be constructed without the use of inhomogeneous or domain-dependent perturbations, which answers positively a question posed by Ruan and Tian.
\end{abstract}

\section{Introduction}

\subsection{Fixed-domain curve counts}

Gromov--Witten theory is concerned with counting holomorphic maps from complex curves to a smooth complex projective variety $X$. Given $g, n \in \N_0$ satisfying $2g - 2 + n > 0$ and $A \in H_2(X,\Z)$, let $\cM_{g,n}(X,A)$ be the moduli space of stable maps of arithmetic genus $g$ and homology class $A$, with $n$ marked points. Consider the map
\begin{equation}
    \label{eqn: tau map}
    \tau= \pi \times \mathrm{ev} \colon \cM_{g,n}(X,A) \to \cM_{g,n} \times X^n,
\end{equation}
where $\pi$ encodes stabilization of the domain of a stable map and $\mathrm{ev}$ its values at the $n$ marked points. The Gromov--Witten invariants are defined by pairing the push-forward of the virtual fundamental class of $\cM_{g,n}(X,A)$ by $\tau$ with cohomology classes in $\cM_{g,n} \times X^n$. It is an interesting but difficult question under what conditions on $X,g,n,A$ and the cohomology class these invariants agree with the corresponding geometric counts of holomorphic curves in $X$.

This article focuses on the \textbf{fixed-domain Gromov--Witten invariants}, defined by cohomology classes of the form $\mathrm{PD}[\mathrm{pt}_{\cM_{gn}}] \otimes \gamma$ where $\gamma \in H^*(X^n,\Z)$. The degree of such a class is equal to that of the virtual fundamental class if
\begin{equation}\label{eqn: num constraint}
   2\langle c_1(X), A \rangle +2 r(1-g) = \deg(\gamma) \quad\text{where } r = \dim_\C X,
\end{equation} 
The resulting invariant is a virtual count of stable maps from $n$-marked genus $g$ nodal curves with a \emph{fixed} stabilization subject to constraints specified by the cycle Poincaré dual to $\gamma$.  

Fixed-domain curve counts arise in various areas of algebraic geometry and mathematical physics, and have recently attracted significant interest. One appealing feature of fixed-domain invariants is their computability: their evaluation reduces to a computation in the quantum cohomology of $X$ \cite{BP, Cela}. Moreover, there is an expectation that they are more often enumerative than the general Gromov--Witten invariants. 

For Grassmannians, the fixed-domain Gromov--Witten invariants are computed using the celebrated Intriligator--Vafa formula  \cite{st,bdw, B1,mo}; see also \cite{mop, BP, Cela}. When $\gamma = \mathrm{PD}[\mathrm{pt}_{X^n}]$, these invariants are also known as \textbf{virtual Tevelev degrees} and denoted by $\vTev_{g,n}(X,A)$.  The systematic study of Tevelev degrees for general targets began with \cite{CPS}, motivated by work on scattering amplitudes in mathematical physics \cite{tev}. This work sparked a series of studies \cite{BLLRST, BP, Cela, CL2, CL, FL, lian_pr, LP} connecting the problem to other areas such as interpolation problems, semistability of the tangent bundle of $X$ \cite{cl_hirzebruch}, and tropical geometry \cite{CIL, cd}. .

Lian and Pandharipande showed that when $\langle c_1(X), A \rangle$ is large, there is also a \textbf{geometric Tevelev degree} $\Tev_{g,n}(X,A)$ defined as the number of points in the general fiber of the map $\tau$ restricted to $\M_{g,n}(X,A)$ \cite{LP}. While, in general, $\vTev_{g,n}(X,A) \neq \Tev_{g,n}(X,A)$  \cite{BP}, Lian and Pandharipande conjectured that  $\vTev_{g,n}(X,A) = \Tev_{g,n}(X,A)$ if $X$ is a smooth Fano variety and $\langle c_1(X), A \rangle$ is sufficiently large \cite{LP}. 
This conjecture has been verified in many cases, such as Del Pezzo surfaces \cite{CL2}, projective spaces \cite{FL, BP, CPS, LP}, homogeneous spaces \cite{LP}, and low-degree complete intersections \cite{LP,BLLRST}.  However, the recent work~\cite{BLLRST} provides explicit counterexamples, such as Fano splitting projective bundles over $\mathbb{P}^k$ for $k > 1$, and certain Fano hypersurfaces in projective space where the boundary of the moduli space of maps dominates the target of the map $\tau$ in~\eqref{eqn: tau map}. These examples raise the question of whether the Lian--Pandharipande conjecture can be refined to account for the discrepancy between the virtual and geometric Tevelev degrees. More generally, they motivate the problem of understanding the geometric meaning of fixed-domain Gromov–Witten invariants of Fano manifolds when $\langle c_1(X), A\rangle$ is large. 

\subsection{The Gromov--Witten pseudocycle}

The failure of the Lian--Pandharipande conjecture shows that the fixed-domain Gromov--Witten invariants do not always agree with counts of holomorphic curves. However, as we show in \autoref{thm: main 0} below, they have an enumerative interpretation in symplectic geometry as \emph{signed} counts of \emph{pseudo-holomorphic} curves.

Our result is related to a question posed by Ruan--Tian \cite{ruan1}, which we now explain. Let $(X,\omega)$ be a compact symplectic manifold which is \textbf{positive} meaning $\langle c_1(X,\omega), A \rangle > 0$
for all $A \in H_2(X,\Z)$ with $\langle [\omega], A \rangle > 0$; the main examples being smooth Fano varieties. For positive (more generally, semipositive) manifolds, Ruan--Tian defined the fixed-domain Gromov--Witten invariants by counting, with signs, smooth maps $u \colon C \to X$ from a fixed smooth $n$-marked genus $g$ curve $(C,\bp=(p_1,\ldots,p_n))$ which solve the \textbf{perturbed Cauchy--Riemann equation}:
\begin{equation}
    \label{eqn: inhomogeneous perturbation}
    \overline\partial_J(u)(z) = \nu(z, u(z))
\end{equation}
Here $J$ is a generic almost complex structure on $X$ compatible with $\omega$, $\delbar_J$ is the induced Cauchy--Riemann operator, and $\nu$ is a generic section of the bundle of complex-antilinear homomorphisms $TC \to TX$ over $C \times X$.  (An alternative approach of McDuff--Salamon is to perturb the fibration $C \times X \to C$ and count pseudo-holomorphic sections \cite[Chapter 8]{mcduff}.)

To show that such a count is well-defined and independent of $(J,\nu)$, Ruan--Tian construct the \textbf{Gromov--Witten pseudocycle}. Denote by  $\M(C,\bp; X,A,J,\nu)$ the space of solutions to \eqref{eqn: inhomogeneous perturbation} representing class $A$ . The advantage of using the perturbed equation \eqref{eqn: inhomogeneous perturbation} is that $\M(C,\bp; X,A,J,\nu)$ is a smooth manifold of expected dimension for generic $(J,\nu)$, and so are all the strata of its Gromov--Kontsevich compactification $\cM(C,\bp; X,A,J,\nu)$. Positivity then implies that the evaluation map
\begin{equation}
    \label{eqn: evaluation map intro}
    \mathrm{ev} \colon \M(C,\bp; X,A,J,\nu) \to X^n
\end{equation}
is a \textbf{pseudocycle}: its image has compact closure, and its limit set has real dimension at most $\dim_\R \M(C,\bp; X,A,J,\nu) - 2$. Such a pseudocycle defines a class in $H_*(X^n,\Z)$ which agrees with the pushforward of the virtual fundamental class of $\cM(C,\bp; X,A,J,\nu)$ \cite[Section 6]{hirshi}. Standard transversality theory implies that the pairing of this class with $\gamma \in H^*(X^n,\Z)$ is a signed count of solutions to \eqref{eqn: inhomogeneous perturbation} whose image intersects a cycle in $X^n$ Poincaré dual to $\gamma$. 

Ruan and Tian asked in \cite[Remark 2.10]{ruan1} whether the Gromov--Witten pseudocycle \eqref{eqn: evaluation map intro} can be realized by the moduli space of $J$--holomorphic maps: solutions to \eqref{eqn: inhomogeneous perturbation} with inhomogeneous term $\nu=0$. We have inclusions
\[
\{ \text{complex structures} \} \subset \{ \text{almost complex structures} \} \subset \{ \text{Ruan--Tian perturbations \eqref{eqn: inhomogeneous perturbation}} \}.
\]
While complex structures are too rigid to define the pseudocycle, it is an interesting question whether we can count pseudo-holomorphic curves rather than solutions to \eqref{eqn: inhomogeneous perturbation}. This is, in general, a hard problem; see \cite[Remark 2.10]{ruan1} and \cite{ionel2, ionel3, zinger2, zinger3}. The key difficulty is that when $\nu = 0$, the space of stable pseudo-holomorphic maps
\[
    \cM(C,\bp;X,A,J) \coloneq \cM(C,\bp; X,A,J,\nu=0)
\] 
has strata of incorrect dimension, consisting of maps that are constant or multiply covered on some components. In fact, this is already true for $\M(C,\bp;X,A,J)$, which decomposes into simple maps $\M^*(C,\bp;X,A,J)$ and multiple covers.  Even with the positivity condition, these strata often do not have real codimension $\geq 2$. Nonetheless, we show that for high degree curves, the evaluation map \eqref{eqn: evaluation map intro} for $\nu=0$ still defines a pseudocycle.

\begin{theorem}\label{thm: main 0}
    Let $(X,\omega)$ be a compact positive symplectic manifold of real dimension $2r\geq 6$. Denote by $\cJ$ the space of $C^\ell$ almost complex structures on $X$ compatible with $\omega$, and for $J \in \cJ$ denote by $\M^*(C,\bp,X,A,J)$ the moduli space of simple $J$-holomorphic maps $C \to X$ from an $n$-marked genus $g$ curve $(C,\bp)$ representing a homology class $A \in H_2(X,\Z)$. 

    There exist constants $c = c(r,g)$ and $d=d(r)$ such that for a generic almost complex structure $J \in \cJ$ and a generic marked curve $(C,\bp) \in \M_{g,n}$ the evaluation map
    \[
        \mathrm{ev} \colon \M^*(C,\bp;X,A,J) \to X^n
    \]
    is a pseudocycle whenever $\langle c_1(X), A \rangle \geq c$ and $\mathrm{codim}(\mathrm{ev}) \leq dn$. 
\end{theorem}

\begin{remark}
    ~
    \begin{enumerate}
    \item By generic we mean: from a subset containing a countable intersection of open dense subsets; such a subset is dense by the Baire category theorem. While \autoref{thm: main 0} is stated for the space of $C^\ell$ of almost complex structures, a standard argument by Taubes allows one to extend it to the space of $C^\infty$ structures \cite{mcduff}. 
    \item
    This result is a special case of \autoref{thm: main II}, where the constant $d(r)$ is made explicit; the constant $c(r,g)$ can be made explicit as well. It is an interesting question whether this bound on $\mathrm{codim}(\mathrm{ev})$ is sharp.
    \end{enumerate}
\end{remark}

\begin{corollary}
    In the situation of \autoref{thm: main 0}, let $\gamma \in H_*(X^n,\Z)$ be a class of degree $\dim_\R \M^*(C,X,J;A)$ and let $f \colon V \to X^n$ be a pseudocycle Poincaré dual to $\gamma$. For a generic $J \in\cJ$, the fiber product of $\mathrm{ev}$ and $f$ is a finite set, and the fixed-domain Gromov--Witten invariant associated with $\gamma$ is equal to the signed count of points in this fiber product.
\end{corollary}


Interestingly, the condition that $\langle c_1(X), A \rangle$ is large is essential in our proof of \autoref{thm: main 0}, but plays no role when using Ruan--Tian perturbations. This condition appears also in algebraic geometry. While the fixed-domain Gromov--Witten invariants of Fano varieties are not always enumerative — that is, $\vTev_{g,n}(X,A) \neq \Tev_{g,n}(X,A)$ in general, even for large $\langle c_1(X), A \rangle$ — they are found to be positive when $\langle c_1(X), A \rangle$ is large in all known examples \cite{BP, Cela, CL2}.  Compare, for instance, \cite[Theorem 1.13]{BLLRST} with \cite[Theorem 5.19]{BP} for examples in which positivity holds, but enumerativity remains unknown.

\begin{conjecture}
    Let $X$ be a smooth Fano variety of complex dimension $r$. The fixed-domain Gromov--Witten invariants are positive when $\langle c_1(X), A \rangle $ is large compared to $r$ and $g$.
\end{conjecture}

It is possible that the positivity of the fixed-domain Gromov--Witten invariant is related to their interpretation as counts of  \emph{unperturbed} pseudo-holomorphic curves.

\subsection{Stratification of the moduli spaces}

\autoref{thm: main 0} is a consequence of a much more general result which analyzes the subsets 
\[
    \M_\G(X,A,J) \subset \cM_{g,n}(X,A,J)
\] 
consisting of stable maps whose domain is modelled on an $n$-marked genus $g$ graph $\G$. These subspaces are not, in general, smooth and of expected dimension. For example, if $\G = T_{g,n}$ is the graph consisting of a single vertex of genus $g$ with $n$ markings, then $\M_\G(X,A,J) = \mathcal{M}_{g,n}(X,A,J)$ can be further decomposed into simple maps and multiple covers. Similarly, for a general graph $\G$ we have the open subset
$
\M_\G^*(X,A,J) \subset \M_\G(X,A,J)
$
of simple maps, which is smooth and of expected dimension for a generic $J$. 

Let $\M_\G(X,A,\cJ) \to \cJ$ be the universal space of stable maps over $\cJ$. For $(C,\bp) \in \M_{g,n}$ denote by $\mathcal{M}_{\G}(C, \underline{\mathsf{p}};X,A,\cJ)$ the fiber of the stabilization map $\pi_{\G}: \cM_\G(X,A,\cJ) \to \cM_{g,n}$ over $(C, \underline{\mathsf{p}})$. Combine the evaluation map \eqref{eqn: tau map} with the projection to $\cJ$ to define
\begin{equation}\label{eqn: map tau}
\mathrm{ev}_{\G}:\mathcal{M}_{\G}(C, \underline{\mathsf{p}}; X, A,\cJ) \to X^n \times \cJ;
\end{equation}

\begin{theorem}\label{thm: main} 
    Let $(X,\omega)$ be a compact positive symplectic manifold of real dimension $2r\geq 6$. There exist constants $c = c(r,g)$ and $d=d(r)$ with the following property. Suppose that $\langle c_1(X), A \rangle \geq c$ and the virtual codimension $2k$ of $\mathrm{ev} \colon \M(C,\bp;X,A,J) \to X^n$ is at most $dn$.
    For a generic marked curve $(C,\bp) \in \M_{g,n}$,  
    \begin{itemize}
    \item[(i)] if $\mathsf{\Gamma} \neq T_{g,n}$, then the image of the map $\mathrm{ev}_{\mathsf{\Gamma}}$ in \eqref{eqn: map tau} has codimension $\geq 2k+2$;
    \item[(ii)] if $\mathsf{\Gamma} = T_{g,n}$, then the the image under $\mathrm{ev}_\G$ of the space of multiple covers
    \[
     \mathcal{M}_{\mathsf{\Gamma}}(C, \underline{\mathsf{p}}; X, A, \cJ) \setminus \mathcal{M}_{\mathsf{\Gamma}}^*(C, \underline{\mathsf{p}} ; X, A, \cJ)
    \]
     has codimension $\geq 2k+2$. 
\end{itemize}
\end{theorem}

Recall that a subset $S \subset B$ of a Banach manifold $B$ is said to have codimension at least $\ell$ if there exists a second countable Banach manifold $E$ and a Fredholm map $f \colon E \to B$ such that $\mathrm{index}(f) \leq -\ell$ and $S \subset f(E)$. 

The main ideas behind the proof of \autoref{thm: main} are as follows.

\begin{enumerate}[(i)]

\item The first step, discussed in \autoref{sec: stratification}, is to further stratify the moduli spaces of maps by introducing the notion of an \textbf{augmented graph} $\cG$, which enhances a dual graph $\G$ of the domain curve with certain additional combinatorial data. This data records the degree of the stable map on each component, the homology class of the underlying simple map, whether two components have the same image, and auxiliary weights $\bm$ associated with each vertex. In particular, we define the associated \textbf{weighted homology class} $[\G, \underline{\mathsf{m}}]$.

\item In \autoref{sec: simplification process} we describe a \textbf{simplification process}: given a map $u$ modelled on an augmented graph $\cG$, we define a simple map $u^s$ modelled on a new augmented graph $\cG^s$. This construction is inspired by \cite[Section~6.1]{mcduff}, but  is substantially more refined. In particular, the genus and homology class of the map may change. However, the weighted homology class remains unchanged.

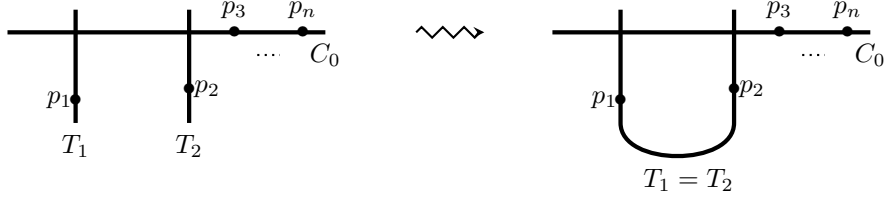
\begin{figure}[h!]
\centering
\begin{tikzpicture}[xscale=0.3,yscale=0.3,
    >=stealth,
    decoration={zigzag, amplitude=2, segment length=8}]

\begin{scope}
  \draw[ultra thick] (5,7) -- (19,7);
  \node at (15,7) {$\bullet$};
  \node at (15,8) {$p_3$};
  \node at (18,7) {$\bullet$};
  \node at (18,8) {$p_n$};
  \node at (19,6) {$C_0$};
  \draw[thick,dotted] (16,6) -- (17,6);

  \draw[ultra thick] (8,8) -- (8,3);
  \node at (8,4) {$\bullet$};
  \node at (7.3,4) {$p_1$};
  \node at (8,2) {$T_1$};

  \draw[ultra thick] (13,8) -- (13,3);
  \node at (13,4.5) {$\bullet$};
  \node at (13.8,4.5) {$p_2$};
  \node at (13,2) {$T_2$};
\end{scope}

\draw[thick,->,decorate] (23,7) -- (26,7);

\begin{scope}[shift={(24,0)}]
  \draw[ultra thick] (5,7) -- (19,7);
  \node at (15,7) {$\bullet$};
  \node at (15,8) {$p_3$};
  \node at (18,7) {$\bullet$};
  \node at (18,8) {$p_n$};
  \node at (19,6) {$C_0$};
  \draw[thick,dotted] (16,6) -- (17,6);

  \draw[ultra thick] (8,8) -- (8,3)
                     to[out=-90,in=-90] (13,3)
                     -- (13,8);

  \node at (8,4) {$\bullet$};
  \node at (7.3,4) {$p_1$};
  \node at (13,4.5) {$\bullet$};
  \node at (13.8,4.5) {$p_2$};

  \node at (11,0.5) {$T_1=T_2$};
\end{scope}

\end{tikzpicture}
\caption{An example of the simplification process in \S\ref{sec: simplification process}, where two separate components $T_1$, $T_2$ (left) are identified to form a single one $T_1 = T_2$ (right).}
\label{fig:simplification-step}
\end{figure}

As an example, suppose $u \colon C \to X$ is a $J$-holomorphic map from a curve $C$ with three components $T_1,T_2$ and $C_0$, as on the left in \autoref{fig:simplification-step}.
Suppose that two rational tails $T_1$ and $T_2$ have the same image curve under $u$ and that $u|_{T_1}$ and $u|_{T_2}$ have degree $1$ onto their images, but the points $u(p_1)$, $u(p_2)$, $u(T_1 \cap C_0)$, and $u(T_2 \cap C_0)$ are all distinct. Then the domain curve of the associated simple map $u^s \colon C^s \to X$ (on the right in \autoref{fig:simplification-step}) has only two components, and $C^s$ has arithmetic genus one. This differs from the procedure in \cite[Section 6.1]{mcduff}, which produces a map of arithmetic genus zero.

The McDuff--Salamon simplification is insufficient for our purposes, as it does not retain enough geometric information. In the above example, it loses the information on which point on $T_1$ maps under $u$ to the same point in $X$ as $T_2 \cap C_0$. Our refinement preserves this data and yields a moduli space whose expected dimension is smaller: this finer control of the dimension is essential for our arguments.

\item   The proof of \autoref{thm: main}  involves only augmented graphs of a special kind, corresponding to stable maps in the fiber of the map $\tau$ and their simplifications. In \autoref{sec: stratification}, we introduce the \textbf{fixed-domain constraint} for augmented graphs. For example, observe that unless $h_1(\G) = 0$, the stabilized domain curve of any map in $\M_{\G}(X, A, \cJ)$ lies in the boundary $\cM_{g,n} \smallsetminus \mathcal{M}_{g,n}$, and hence $\M_{\G}(C, \underline{\mathsf{p}}; X, A, \cJ)$ is empty. Thus, we may assume that the graph $\G$ contains a unique vertex of genus $g$, which we label by $0$, and that each connected component of $\G \smallsetminus \{0\}$ contains at most one marking. 

The class of augmented graphs $\G$ with $\G$ of such shape is not stable under the simplification procedure introduced in \autoref{sec: simplification process}. Augmented graphs satisfying the fixed-domain constraint form a class that, with some exceptions, is closed under simplification and includes the graphs of the simple form described above.

\item  The most technical part of the paper is \autoref{sec: proof of main II} which proves  \autoref{thm: main II}: a refined version of \autoref{thm: main} for augmented graphs satisfying the fixed-domain constraint. Via the simplification process, we reduce the proof to the situation where all maps in the moduli space either are simple or fail to be simple due to the main component being multiply covered, constant, or having the same image as another component. These are the most delicate cases and each of them is treated separately using different methods.

\item In \autoref{subsec: thm II to I} we derive the main results from \autoref{thm: main II}.  Finally, for completeness we include in \autoref{sec: transversality} a short proof of the transversality theorem for simple maps modelled on an arbitrary genus $g$ graphs (proved for genus zero in \cite{mcduff} and for arbitrary genus using inhomogeneous perturbations in \cite{ruan1,ruan2}).
\end{enumerate}

The techniques developed in this paper, especially the decomposition of the moduli spaces of stable maps into a finer stratification that records additional combinatorial data, along with the refined simplification process that preserves this data, are quite general and not limited to the study of fixed-domain Gromov–Witten invariants. We expect these methods to have further applications in enumerative geometry, particularly in problems where the domain curve is required to satisfy specified constraints.

\subsection{Acknowledgments}

The question of enumerativity for almost complex structures was raised by the first author in a conversation with Roya Beheshti, following her talk at the NSF-funded \emph{Workshop on Tevelev Degrees and Related Topics} at the University of Illinois Urbana-Champaign. In her presentation, Roya explained a negative result for the analogous enumerative problem in the algebraic setting. We are grateful to the workshop organizers, Felix Janda and Deniz Genlik, and to the speakers for fostering a stimulating research environment that led to this project.

The first version of this paper did not discuss earlier work of Ruan--Tian and McDuff--Salamon on the fixed-domain invariants, and we are grateful to Spencer Catallani, Kai Cieliebak, and Dominic Joyce for bringing it to our attention. We thank also Roya Beheshti, Carl Lian, Miguel Moreira, Rahul Pandharipande, John Pardon, Dhruv Ranganathan, and Aleksey Zinger.

The first author is supported by the SNF grant P500PT-222363. The second author is supported by Trinity College, Cambridge, an he thanks Chiu-Chu Melissa Liu and Francesco Lin for hosting him at Columbia University, where part of this work was carried out.

\section{Refined stratification} 
\label{sec: stratification}

The proof of \autoref{thm: main} involves decomposing $\M_{\mathsf{\Gamma}}(X, A, \cJ)$ into strata whose codimension in $\cJ$ is controlled by a rather subtle induction with respect to the complexity of $\G$. For the inductive argument it is important to consider only graphs of a particular shape and to augment them with certain additional discrete data corresponding to the stratification of $\M_{\mathsf{\Gamma}}(X, A, \cJ)$.

\subsection{Augmented graphs}\label{sec: graphs}

We start recalling the standard definition of prestable graphs. 

\begin{definition}\label{def: stable graph}
An $n$-marked genus $g$ \textbf{prestable graph} $\mathsf{\Gamma}$ is a tuple
\[
\mathsf{\Gamma} = (V, H, \bg, v, \iota, \bp) 
\]
where
\begin{enumerate}[(i)]
    \item $V = V(\G)$ is a finite set of vertices, with a function $\bg \colon V \to \N_0$ assigning a genus to each vertex. We will write $g_\alpha$ in place of $\bg(\alpha)$; 
    \item  $H = H(\G)$ is a finite set of half-edges, with a map $v \colon H \to V$ assigning an incident vertex to each half-edge;
    \item  $\iota \colon H \to H$ is an involution  whose fixed points 
    \[
        P = P(\G) = \{ h \in H \ | \ \iota(h) = h\}
    \]
    form the set of markings and pairs 
    \[
        E = E(\G) = \{ \{ h, h' \} \subset H \ | \ h' = \iota(h), \ h\neq h' \}
    \]
    form the set of unoriented edges of $\G$;
    \item  $\bp : \{1, \ldots, n\} \to P$ is a bijection defining the marking. We will write $p_i$ in place of $\bp(i)$. 
    
\end{enumerate}
We will always assume that $\G$ is connected. The \textbf{genus} of $\G$ is 
\[
    g(\G) = \sum_{\alpha \in V} g_\alpha +h_1(\G).
\]
Finally, $\G$ is said to be \textbf{stable} if for all $\alpha \in V$ it satisfies
\[
    2g_\alpha - 2 + \mathrm{val}(\alpha) > 0,
\]
where the \textbf{valence} of $\alpha$ is defined by $\mathrm{val}(\alpha) = |\{ h \in H : v(h) = \alpha \}|$.

\end{definition}

\begin{example}
    The dual graph of a stable $n$-marked genus $g$ curve $(C,\bp)$ representing a point in the Deligne--Mumford space $\cM_{g,n}$ is a stable $n$-marked genus $g$ graph.    (By a slight abuse of notation, we will use the notation $\bp$ and $p_i$ to denote both the marking of the graph $\G$ and the corresponding marked points on the curve $C$ modelled on $\G$.)
\end{example}

We are interested in the stratification of the moduli space $\M_{\mathsf{\Gamma}}(X, A, \cJ)$ of maps $u \colon C \to X$ modelled on $\G$. Therefore, it will be important to keep track of additional data encoding the stratum containing $u$, such as the degree of each map $u_\alpha$ on its image, the homology class $[u_\alpha]$, as well as certain auxiliary weights $m_\alpha$. To this end, we enrich $\G$ with such additional data.

\begin{definition}\label{def: augmented graphs}
    An $(n+\ell)$-marked genus $g$ \textbf{augmented graph} is a tuple
    $$
    \widetilde{\mathsf{\Gamma}} = (\mathsf{\Gamma}, \underline{\mathsf{m}}, \underline{\mathsf{d}}, \underline{\mathsf{A}}, h).
    $$
    where
\begin{enumerate}[(i)]
    \item $\mathsf{\Gamma}$ is an $(n+\ell)$-marked genus $g$ prestable graph with an order on the set of vertices; the smallest vertex will be denoted by $0 \in V(\G)$; 
    \item The $(n+\ell)$-marking consists of an $n$-marking $\bp = (p_1,\ldots,p_n)$ and an $\ell$-marking $\bp' = (p_1', \ldots, p_\ell')$ where $0 \leq \ell \leq n$;
    \item $\underline{\mathsf{m}}$ and $\underline{\mathsf{d}}$ are vectors of weights $m_\alpha \in \N_0$ and degrees $d_\alpha \in \N_0$, indexed by $\alpha \in V(\G)$;
    \item $\underline{\mathsf{A}}$ is a vector of homology classes $A_\alpha \in H_2(X, \mathbb{Z})$ indexed by $\alpha \in V(\G)$ such that either $A_\alpha$ is positive or $A_\alpha = 0$, with the latter if and only if $d_\alpha = 0$ if and only if $ m_\alpha = 0$;
    \item $h$ is a function $h: \{ (\alpha, \alpha') \in V(\mathsf{\Gamma})^{\times 2} \mid \alpha \neq \alpha', \ d_\alpha, d_{\alpha'} > 0 \} \to \{ 0, 1 \}$, which will encode whether two components of the domain are mapped to the same image.
\end{enumerate}
We will call $\G$ the \textbf{underlying graph} of $\cG$ and $(\bm,\bd,\bA,h)$ the \textbf{augmentation data} of $\cG$.  The \textbf{homology class} of $\cG$ is 
\[
    [\cG] = \sum_{\alpha\in V} d_\alpha A_\alpha
\]
and the \textbf{weighted homology class} of $\cG$ is
\[
    [\cG,\bm] = \sum_{\alpha\in V} m_\alpha d_\alpha A_\alpha.
\]
\end{definition}

\begin{remark}
    When the number of marked points and the genus are clear from context, we will simply refer to $\cG$ as an augmented graph.
\end{remark}

In addition, the following properties of augmented graphs will be crucial in the proof of \autoref{thm: main}. From a geometric point of view, we are only interested in stable maps in the general fiber of \eqref{eqn: tau map} and the maps obtained by their simplification. This imposes additional constraints on the graph and augmentation data. 

\begin{definition}\label{def: condition *}
    We say that $\cG$ satisfies the \textbf{fixed-domain constraint} if:
    \begin{enumerate}[(i)]
        \item The underlying graph $\G$ has the following shape. The smallest vertex $0 \in V(\G)$ has genus $g$ while all other vertices have genus $0$. In particular, $g(\G) = g + h_1(\mathsf{\Gamma})$.  
        \item  The markings $\bp$ and $\bp'$ are constrained as follows. Vertex $0$ carries the last $n - \ell$ markings $p_{\ell+1},\ldots, p_n$ and is connected via a single edge to $\ell$ vertices of $\alpha_1,\ldots,\alpha_\ell$ of genus $0$, each carrying the corresponding marking $p_1, \ldots, p_\ell$ and satisfying $\mathrm{val}(\alpha_i) = 3$.  The remaining markings in $\bp'$ are carried by vertices other than $0$ and $\alpha_1,\ldots,\alpha_\ell$; 
        \item Every loop of vertices in $\G$ contains a vertex $\alpha$ with $d_\alpha > 0$;
        \item The number of edges connecting each connected component of $\G \smallsetminus \{ 0 \}$ to $0$ is equal to the number of markings in $\bp'$ on that component. In particular, $|E(\G)| \geq 2\ell$. 
    \end{enumerate}
\end{definition}

With two notable exceptions that will be explained later, all augmented graphs appearing in this paper will satisfy the fixed-domain constraint.

\subsection{Stable maps modelled on augmented graphs}

Recall that an \textbf{$n$-marked stable map} consists of a nodal curve $C$, an ordered set $\bp$ of $n$ distinct points in the smooth locus of $C$, and a pseudo-holomorphic map $u \colon C \to X$ such that the automorphism group of the triple $(C,\bp, u)$ is finite. 

\begin{definition}\label{def: stabilization morphism}
    The stabilized domain curve associated to the triple $(C, \bp, u)$ is the $n$-marked stable curve $(\bar C, \bar \bp)$, constructed as follows. The curve $\bar C$ is obtained by contracting each genus-zero component of $C$ that carries fewer than three special points (i.e., nodes or markings). If such a contracted component contains a marking $p_i$, the point to which it is contracted becomes the new marking $\bar p_i$. If $p_i$ lies on a component that is not contracted, then $\bar p_i = p_i$. For the definition of stabilization in families, see \cite[Chapter 10, Section 8]{ACGH}.
\end{definition}

The map $\pi: \cM_{g,n}(X,A) \to \cM_{g,n}$ in Equation \eqref{eqn: map tau} sends a point $[u,C,\bp]$ to the stabilized domain curve $[\bar{C},\bar{\bp}]$.

\begin{notation}
    Let $A \in H_2(X,\Z)$ be a positive class.
    Let $\G$ be an $n$-pointed genus $g$ graph. We denote by
    \[
    \mathcal{M}_\mathsf{\Gamma}(X, A, \cJ)
    \]
    the universal moduli space of marked stable maps modelled on $\G$, that is the space parametrizing isomorphism classes of data $(C,\bp, u,J)$ where
    \begin{enumerate}[(i)]
        \item $J \in \cJ$, 
        \item $C$ is a nodal curve and $\bp = (p_1,\ldots, p_n)$ is an ordered set of smooth points in $C$ such that the marked nodal curve $(C,\bp)$ is modelled on $\G$; that is:  $\G$ is the dual graph of $(C,\bp)$.
        \item $u \colon (C,\bp) \to X$ is a $J$-holomorphic stable map such that $[u] \coloneq u_*[C] = A$. 
    \end{enumerate}
    For every vertex $\alpha$ of $\G$ denote by $C_\alpha$ the corresponding irreducible component of $C$ and by $u_\alpha = u|_{C_\alpha}$ the restriction of $u$ to $C_\alpha$. 
\end{notation}

\begin{example}
    If $\G = T_{g,n}$ consists of a single vertex of genus $g$ with $n$ markings, then
    \[
        \mathcal{M}_\mathsf{\Gamma}(X, A, \cJ) = \mathcal{M}_{g,n}(X, A, \cJ)
    \]
    is the universal moduli space of pseudo-holomorphic maps to $X$ from genus $g$ smooth curves with $n$ marked points in class $A$. The corresponding Gromov--Kontsevich universal moduli space of stable maps is
    \[
        \overline\M_{g,n}(X, A, \cJ) = \bigcup_\G \mathcal{M}_\G (X, A, \cJ)
    \]
    where the union is taken over all $n$-marked genus $g$ graphs $\G$.
\end{example}

The next is a higher genus generalization of \cite[Definition~6.1.1]{mcduff}, and will be a central notion in what follows.

\begin{definition}
    \label{def: simple}
    A stable map $(C,\bp, u)$ is called \textbf{simple} if it satisfies the following conditions:
    \begin{enumerate}[(i)]
        \item for every connected (possibly reducible) subcurve $C' \subset C$ of positive arithmetic genus, the restriction $u|_{C'}$ is nonconstant;
        \item each nonconstant component $u_\alpha$ is not multiply covered;
        \item if $u_\alpha$ is nonconstant and $\alpha \neq \beta$, then $\im(u_\alpha) \neq \im(u_\beta)$. 
    \end{enumerate}
    We denote by 
    $$
    \M_\mathsf{\Gamma}^*(X,A,\cJ) \subset \M_\mathsf{\Gamma}(X,A,\mathcal{J})
    $$ 
    the subset of simple maps.
\end{definition}

\begin{notation}
Recall that an ordered multiset of $n$ elements of a set $S$ is a sequence $(a_1,\ldots, a_n)$ where $a_i \in S$ and repetitions are allowed.  In this paper, we will use both ordered sets and ordered multisets of marked points, and the distinction will be important. 
\end{notation}

We will use the following generalization of this notion introduced in \cite[Section 6.1]{mcduff}. 

\begin{definition}
    \label{def: weighted prestable map}
    A \textbf{weighted $n$-marked prestable map} $(C, \bp, u, \bm)$ consists of
    \begin{enumerate}[(i)]
        \item a curve $C$ (not necessarily nodal),
        \item an ordered multiset of $n$ points in $C$ (not necessarily contained in the smooth locus),
        \item a vector $\bm$ of weights $m_\alpha \in \N_0$ indexed by the irreducible components $C_\alpha$ of $C$,
        \item a pseudo-holomorphic map $u \colon C \to X$,
    \end{enumerate}
    with the property that $(C,\bp,u)$ is obtained from an $n$-marked stable map $(\tilde C, \tilde\bp, \tilde u)$ by identifying, finitely many times, some number of $\P^1$ components which have the same image under $\tilde u$. Each such collection is identified to a single point in the resulting curve.

    A weighted $n$-marked prestable map $(C,\bp,u,\bm)$ is said to be \textbf{stable} if the underlying $n$-marked map $(C,\bp,u)$ is stable. Similarly, $(C, \bp, u, \bm)$ is called \textbf{simple} if $(C, u)$ is simple.

    Associated with such a map is the usual homology class
    \[
        [u] = \sum_\alpha [u_\alpha]
    \]
    where $u_\alpha = u|_{C_\alpha}$, as well as the \textbf{weighted homology class} is defined by
    \[
        [u,\bm] =  \sum_\alpha m_\alpha [u_\alpha].
    \]
\end{definition}

\begin{notation}
    When the number of marked points is clear from context, we will simply refer to $(C,\bp,u,\bm)$ as a weighted prestable (or stable) map. 
\end{notation}

\begin{remark}
    It follows from the definition that the underlying curve $C$ of a weighted prestable map has ordinary singularities. In particular, as a topological space, $C$ is the quotient of a smooth curve, its normalization, by an equivalence relation that identifies finitely many finite sets of points, each to a single point.
\end{remark}

\begin{definition}  
    \label{def: augmented maps}
    Let $\cG$ be an augmented graph as in \autoref{def: augmented graphs}. 
    A weighted $(n+\ell)$-marked stable map $(C,\bp\cup\bp', u,\bm)$ is \textbf{modelled on $\cG$} if 
    \begin{enumerate}[(i)]
        \item $\G$ is the dual graph of $(C,\bp \cup \bp')$; the component $C_0$ corresponding to the smallest vertex $0 \in V(\G)$ will be called the \textbf{main component}; 
        \item $\bm$ agrees with the vector of weights of $\cG$;
        \item $u_\alpha$ is constant if and only if $d_\alpha = 0$ if and only if $m_\alpha=0$, and otherwise $u_\alpha$ is a $d_\alpha$-fold cover of a simple map representing homology class $A_\alpha$; in particular:  $[u_\alpha] = d_\alpha A_\alpha$;
        \item $\im(u_\alpha) = \im(u_\beta)$ if and only if $h(\alpha,\beta) = 1$.
    \end{enumerate}
    Denote by 
    \[
        \M_{\cG}(X,\cJ)
    \]
    the universal moduli space of weighted $(n+\ell)$-marked stable maps modelled on $\cG$.
\end{definition}

\begin{remark}
    The order on the vertices of $\G$ does not play any role in the definition of $\M_{\widetilde{\mathsf{\Gamma}}}( X, \mathcal{J})$, but they will play a role in our induction later.
\end{remark}

\subsection{Generalization of the main theorem}\label{sec: generalization of the main theorem}

Let $\cG$ be an augmented graph as in \autoref{def: augmented graphs} which satisfies the fixed-domain constraint from \autoref{def: condition *} and with weighted homology class satisfying

\begin{equation}\label{eqn: num constraint inequality}
    c_1([\cG,\bm])= r(n+g-1)-k
\end{equation}
where we assume $k \in \N_0$ to satisfy
\begin{equation}
    \label{eqn: bound on k} 
    k \leq  \frac{r-1-\rho}{5} n \quad\text{for some } \rho >0.
\end{equation}
The number $2k$ is the virtual relative dimension of the map \eqref{eqn: tau map}. 

Let $\M_{\cG}(X,\cJ)$ be the universal moduli space introduced in \autoref{def: augmented maps} and consider the map
\begin{equation}\label{eqn: map tau tilde}
\tau_{\widetilde{\mathsf{\Gamma}}}= \pi_{\cG} \times \mathrm{ev}_{\cG}: \M_{\widetilde{\mathsf{\Gamma}}}( X, \mathcal{J})  \to \cM_{g,n} \times X^n \times \cJ 
\end{equation}
defined as follows:
\begin{enumerate}[(i)]
    \item the map $\pi_{\cG}$ on the first factor records the smooth genus $g$ curve corresponding to the vertex $0$ along with the markings $\bp = (p_1,\ldots,p_n)$. Although, for $i \leq \ell$ the marking $p_i$ does not lie in $C_0$, since the graph $\cG$ satisfies the fixed-domain constraint, $p_i$ lies in $ C_{\alpha_i}$ an the map $\tau_{\cG}$ is recalling the intersection point $C_0 \cap C_{\alpha_i}$ as $i$-th marking on $C_0$. Note that this procedure differs from the stabilization process of \autoref{def: stabilization morphism},
    \item on the second factor is the evaluation map at $(p_1',\ldots,p_{\ell}', p_{\ell+1},\ldots, p_n)$,
    \item on the third factor is the projection on the space of almost complex structures on $X$. The map $\mathrm{ev}_{\cG}$ denotes the projection on the second and third factors.
\end{enumerate}

The remaining sections of this paper are devoted to the proof of the following theorem and how it implies \autoref{thm: main}.

\begin{theorem}\label{thm: main II}
    There exists a constant $c = c(r,g)$ with the following property. Let $(X,\omega)$ be a compact positive symplectic manifold of dimension $2r \geq 6$ and let $\cG$ be an augmented graph as above satisfying \eqref{eqn: num constraint inequality} and \eqref{eqn: bound on k} . If $c_1([\cG,\bm]) \geq c$, then the image of the map $\tau_{\widetilde{\mathsf{\Gamma}}}$ in \eqref{eqn: map tau tilde} has codimension $\geq 2k+2$, unless $\G= T_{g,n}$ and $m_{0}=d_{0}=1$.
\end{theorem}

Denoting by
$
\M_{\cG}(C, \bp ;X,\cJ)
$
the fiber the projection $\pi_{\cG}: \M_{\cG}(X,\cJ) \to \cM_{g,n}$ over $(C, \bp)$, we obtain the following fixed-domain reformulation of the previous result.

\begin{corollary}\label{cor: main II}
    In the situation described in \autoref{thm: main II}, for a generic $(C,\bp) \in \cM_{g,n}$, the image of the map
    $$
    \mathrm{ev}_{\widetilde{\mathsf{\Gamma}}}: \M_{\cG}(C, \bp ;X,\cJ) \to X^n \times \cJ
    $$
    has codimension $2k+2$ unless $\G= T_{g,n}$ and  $m_{0}=d_{0}=1$.
\end{corollary}

\begin{proof}
    Let $f: B \to \M_{g,n} \times X^n \times \cJ$ is a Fredholm map of index $d \leq -2k-2$ whose image contains the image of $\tau_{\cG}$. Let $(C,\bp) \in \M_{g,n} \subset \cM_{g,n}$ be the regular value of the composition of $f$ with the projection to $\cM_{g,n}$, so that the fiber $B(C,\bp)$ over $(C,\bp)$ is a Banach manifold. By the Sard--Smale theorem, the set of such regular values is residual in $\M_{g,n}$. Given such $(C,\bp)$, the composition
    \[
    B(C,\bp) \hookrightarrow B \xrightarrow{f} X^n \times \cJ
    \]
    is a Fredholm map of index $d$ whose image contains $\mathrm{ev}_{\cG}(\M_{\cG}(C, \bp; X, \cJ))$.
\end{proof}

\autoref{thm: main II} is significantly more general than \autoref{thm: main}, whose proof follows from the former and is discussed in \autoref{subsec: thm II to I}. The proof of \autoref{thm: main} relies on a special class of augmented graphs and makes limited use of the augmentation data. The reason for considering more complicated graphs in \autoref{thm: main II}, rather than only those appearing in the proof of \autoref{thm: main}, is that the latter do not form a class closed under the simplification process described in the next section. This process makes use of the full augmentation data and is essential to reduce, by induction, the problem to one concerning simple maps, to which we can then apply the transversality statement for moduli spaces of simple maps discussed in \autoref{sec: transversality}.  In other words, the definitions in this section and statement of \autoref{thm: main II} are carefully designed to ensure that the induction stays within the same class of graphs equipped with additional combinatorial data.

\section{Simplification process}\label{sec: simplification process}

This section discusses how every weighted stable map modelled on an augmented graph covers a simpler stable map with the same image and weighted curve class. This construction is inspired by the construction of a simple map underlying a stable genus $0$ map from \cite[Section~6.1]{mcduff}. However, the map obtained by our construction carries significantly more information having to do with the specific shape of graphs considered in \autoref{def: augmented graphs} as well as the augmentation data. In particular, the genus of the new map does not necessarily agree with that of the original map.

\begin{definition}
    Let $(C,\bp)$ be the marked domain of a weighted prestable map as in \autoref{def: weighted prestable map}. In particular,  we do not assume that $C$ is nodal or that the points $\bp$ are distinct or belong to the smooth locus of $C$.
    
    A point $p \in C$ is called a \textbf{destabilizing point for $(C, \bp)$} if one of the following is true:
    \begin{enumerate}[(i)]
        \item $p$ belongs to at least three irreducible components of $C$;
        \item $p$ belongs to $\bp$ and two irreducible components of $C$,
        \item $p$ appears at least twice in $\bp$.
    \end{enumerate}
\end{definition}
 
\begin{definition}\label{def: stabilization process}
    Let $(C,\bp,u,\bm)$ be a weighted prestable map as in \autoref{def: weighted prestable map}. Let $p \in C$ be a destabilizing point for $(C,\bp)$. A \textbf{stabilization at $p$} is the result $(C^{\mathrm{stab}}, \bp^{\mathrm{stab}}, u^{\mathrm{stab}}, \bm^{\mathrm{stab}})$ of a process of replacing $p$ with a new irreducible contracted component $\P^1$, which we now describe.   
    
    Let $\{ C_i^\circ \}$ be the connected components of $C \smallsetminus \{ p \}$, and let $C_i$ be their closures in $C$. Denote by $q_i \in C_i$ the point corresponding to $p$ and by $\bq$ the sequence of these points. Let $\bp^0 \subset \bp$ be the subsequence consisting of all appearances of $p$ in $\bp$.  Choose two distinct points $a,b$ in $\bq \cup \bp^0$. Define a new curve $C^{\mathrm{stab}}$ by attaching each $C_i$ to $\P^1$, with respect to the following identification
    \[
        q_i \sim 
        \begin{cases}
            0 \in \P^1 & \text{if } q_i = a, \\
            1 \in \P^1 & \text{if } q_i = b, \\
            \infty \in \P^1 & \text{otherwise}.
        \end{cases}
    \]
    Similarly, define the ordered multiset $\bp^{\mathrm{stab}}$ in $C^{\mathrm{stab}}$ to be image of $\bp$ under the  map $\iota \colon \bp \to C^{\mathrm{stab}}$ given by
    \[
        \iota(r) = 
        \begin{cases}
            r \in  \bigcup_i C_i^\circ & \text{if } r \in \bp\setminus \bp^0 \\
            0 \in \P^1 & \text{if } r = a, \\
            1 \in \P^1 & \text{if } r = b, \\
            \infty \in \P^1 & \text{otherwise}.
        \end{cases}
    \]
    Irreducible components of $C^{\mathrm{stab}}$ are the irreducible components of $C$ and $\P^1$. We extend $\bm$ to $\bm^{\mathrm{stab}}$ by assigning weight zero to $\P^1$. Finally, the map $u$ extends to a map $u^{\mathrm{stab}}$ by constant $u(p)$ on $\P^1$. 
\end{definition} 

\begin{remark}
    \label{rem: stabilization process}
    The result of a stabilization is a weighted prestable map with fewer destabilizing points. Therefore, by applying this process repeatedly to a weighted prestable map we can construct a weighted stable map. 
\end{remark}

Let $\cG$ be an augmented graph as in \autoref{def: augmented graphs} and let 
\[
    (C,\bp\cup\bp',u,\bm)
\]
be a weighted marked stable map modelled on an augmented graph $\cG$ as in \autoref{def: augmented maps}. Suppose that $u$ is not simple. The construction described below will produce from this data a new augmented graph $\cG^s$ and a new, simpler weighted marked stable map modelled on $\cG^s$,
\[
    (C^s,\bp^s\cup (\bp')^s,u^s,\bm^s).
\]
Importantly, the corresponding homology classes will be related by
\[
    [\cG^s,\bm^s] = [\cG,\bm] \quad\text{and}\quad c_1([\cG^s]) \leq c_1([\cG]),
\]
which will allow us to carry out an induction process.

\begin{definition}\label{def: simplification process}
    A \textbf{partial simplification} of a weighted stable map $(C,\bp\cup\bp',u,\bm)$ is the weighted stable map $(C^s,\bp^s\cup (\bp')^s,u^s,\bm^s)$  obtained by one of the following procedures.

\begin{enumerate}[(a)]
     \item (\textbf{Connected cover}). In this construction, we replace a component $u_\alpha$ with $d_\alpha > 1$ by the underlying simple map. To be more precise, $u_\alpha$ factors as
         $$
         u_\alpha: C_\alpha \xrightarrow{\phi} C_\alpha^s \xrightarrow{u_\alpha^s} X
         $$
     where 
     \begin{enumerate}[(i)]
        \item $C_\alpha^s$ is a smooth complex curve of genus $g_\alpha^s \leq g_\alpha$, 
        \item $\phi$ has degree $d_\alpha$, 
        \item $u_\alpha^s$ is simple.
    \end{enumerate}
    Let $\{ C_i^\circ \}$ be the connected components of $C \smallsetminus C_\alpha$ and $C_i$ their closure in $C$. Let $q_{i,1},\ldots, q_{i,k_i} \in C_\alpha$ be the points of intersection $C_i \cap C_\alpha$. Set
    \begin{align*}
        I &= \{ i \in \{1,\ldots,n\} \ | \ p_i \in C_\alpha \}, \\
        I' &= \{ i \in \{1,\ldots,\ell\} \ | \ p_i' \in C_\alpha \}.
    \end{align*}
    (Note that if $\cG$ satisfies the fixed-domain constraint, then $I=\emptyset$ if $\alpha \neq 0$ and $I'=0$ if $\alpha=0$.) Define a weighted prestable map $(C^{\mathrm{pre}}, \bp^{\mathrm{pre}}\cup(\bp^{\mathrm{pre}})', u^{\mathrm{pre}}, \bm^{\mathrm{pre}})$  as follows:
    \begin{enumerate}[(i)]
        \item  $C^\mathrm{pre}$ is the curve obtained from $C$ by removing $C_\alpha$ and attaching $C_i$  to a new component $C_\alpha^s$ using the identification $q_{i,j} \sim \phi(q_{i,j})$;
        \item  $\bp^{\mathrm{pre}}$ and $(\bp^{\mathrm{pre}})'$ are ordered multisets of points in $C^{\mathrm{pre}}$ obtained from $\bp$ and $\bp'$ by replacing $\{ p_i \}_{i\in I}$ and $\{ p_i'\}_{i\in I'}$ by their images in $C_\alpha^s$ under $\phi$;
        \item $u^{\mathrm{pre}} \colon C^{\mathrm{pre}} \to X$ is obtained from $u$ by replacing $u_\alpha$ by $u_\alpha^s$ on the component $C_\alpha^s$; we equip it with degree $d_\alpha^s=1$ and weight $m_\alpha^s = d_\alpha m_\alpha$. 
    \end{enumerate}
    We then define $(C^s,\bp^s\cup (\bp')^s,u^s,\bm^s)$ as the result of the repeated stabilization process described in \autoref{def: stabilization process} and \autoref{rem: stabilization process}. The resulting weighted stable map is modelled on an augmented graph $\cG^s$ where the component $C_\alpha^s$ has genus $g_\alpha^s$ and the $h^s$ agrees with $h$ in $\cG$.
    
     \item  (\textbf{Disconnected cover.}) In this construction, we identify two simple components $u_\alpha$ and $u_\beta$ with the same image. To be precise, suppose that $h(\alpha, \beta) = 1$ and $d_\alpha = d_\beta = 1$. Since $\im(u_\alpha) = \im(u_\beta)$, there exists an isomorphism $\phi: C_\alpha \to C_{\beta}$ such that $u_{\alpha} = u_{\beta} \circ \phi$. By applying the procedure from case (a) to the disconnected double cover $\phi \sqcup \mathrm{id} \colon C_\alpha \sqcup C_\beta \to C_\beta$, define
     a new weighted prestable map $(C^{\mathrm{pre}}, \bp^{\mathrm{pre}}\cup (\bp^{\mathrm{pre}})', u^{\mathrm{pre}}, \bm^{\mathrm{pre}})$ with $C_\alpha\sqcup C_\beta$ collapsed to $C_\beta$. Equip the new component $u_\beta^s \colon C_\beta \to X$ with  weight $m_\beta^s = m_\alpha + m_\beta$. We then apply the repeated stabilization process to obtain a weighted stable map $(C^s,\bp^s\cup (\bp')^s,u^s,\bm^s)$ modelled on a graph $\cG^s$. The new function $h^s$ is defined on pairs where neither of the two vertices is $\alpha$, and on these pairs, it agrees with the previous $h$. 

    \item (\textbf{Contracted main component.}) In this construction, we assume that $d_0 = 0$, that is: $u$ is constant on the main component corresponding to $\alpha=0$, which we then collapse to a point. To be precise, in this case $u$ factors through a weighted prestable map $u^{\mathrm{pre}}: C^{\mathrm{pre}} \to X$ obtained from $C$ by collapsing $C_{0}$ to a point. Define $(C^s,\bp^s\cup (\bp')^s,u^s,\bm^s)$ to be the result of repeated stabilization applied to this weighted prestable map, with $\bm^s$ obtained from $\bm$ by declaring it  to be zero on the new genus $0$ components added in the stabilization process, and $h^s = h$ unchanged.
\end{enumerate}

    In each of these constructions, we will say that $C_\alpha$  (and $C_\beta$ in case (b)) is the \textbf{component involved in the simplification}. 

\end{definition}

\begin{remark}
    By repeatedly applying the partial simplification process described above to a weighted stable map with no connected contracted subcurve of positive arithmetic genus, one obtains a weighted simple stable map. 
\end{remark}

An important point to make is that the fixed-domain constraint from \autoref{def: condition *} is \emph{not} preserved by the partial the simplification process. However, as explained by the next result, this issue occurs only if the partial simplification involves the main component.

\begin{proposition}
    \label{prop: simplification}
    Let $(C,\bp\cup\bp',\bm)$ be a weighted stable map modelled on an augmented graph $\cG$ satisfying  the fixed-domain constraint. Let $(C^s,\bp^s\cup (\bp')^s,u^s,\bm^s)$ be the result of the partial simplification process described in \autoref{def: simplification process}, with the corresponding augmented graph $\cG^s$. The partial simplification satisfies
    \[
        [\cG^s,\bm^s] = [\cG,\bm] \quad\text{and}\quad c_1([\cG^s]) \leq c_1([\cG])
    \]
    Moreover, $\cG^s$ satisfies the fixed-domain constraint and
    \[
        \tau_{\cG}((C,\bp\cup\bp',\bm)) = \tau_{\cG^s}((C^s,\bp^s\cup (\bp')^s,u^s,\bm^s)),
    \]
    unless the partial simplification involves the main component $C_0$ of $C$. 
\end{proposition}

\begin{proof}  
    It is straightforward to verify from \autoref{def: simplification process} unless the partial simplification involves the main component the partial simplification process preserves all the conditions listed in \autoref{def: condition *}, so let us comment on what goes wrong in these two special cases.  
    
    In case (a), it is possible for the multiple cover $\phi \colon C_0 \to C_0^s$ to map two or more markings in  $p_{\ell+1}, \ldots, p_n \in C_0$, or the points joining $C_{0}$ to a degree $0$ component carrying one of the $p_1,\ldots, p_\ell$ to the same point. Consequently, the curve $C^s$ obtained from the simplification may no longer be modelled on a graph of the form described in items (i) and (ii) of \autoref{def: condition *}. Clearly, the complex domain curve in the image of $\tau_{\cG'}$ is no longer the same either. When $g=0$, a similar issue can arise when case (b) of \autoref{def: simplification process} is applied at the vertex $0$, as the isomorphism $\phi \colon C_\alpha \to C_\beta$  may map one of the points $p_{\ell+1}, \ldots, p_n \in C_0$ to a marking in $\bp'$ on $C_\beta$. 
\end{proof}

\autoref{prop: simplification} shows that the class of stable maps satisfying the fixed-domain constraint is almost closed under partial simplification. This will be used in the inductive step of the proof of \autoref{thm: main II}. The special cases when the partial simplification involves the main component will appear as the base cases of induction. For those we will use different arguments. The following observations will be useful.

\begin{remark}\label{rmk: simplication cover of spine}
    Consider a connected cover simplification described in case (a) of \autoref{def: simplification process} applied to the main component. Let $q_1,\dots,\ q_\ell \in C_0$ be the points where the $\P^1$ components $C_{\alpha_1},\ldots, C_{\alpha_\ell}$ from point (ii) in \autoref{def: condition *} are attached. Let $m$ be the number of distinct images under $\phi$ of the points $p_i$ for $i > \ell$ and $q_i$ for $i \leq \ell$. Let $J \subset \{1,\ldots, \ell\}$ and $K \subset \{\ell + 1, \ldots, n\}$ be subsets of indices such that $|J| + |K| = m$, and the points $p_j$ for $j \in J$ and $q_k$ for $k \in K$ have distinct images under $\phi$. Then, there is a natural map  
    \[
    \beta : \mathcal{M}_{\cG^s}(X, \mathcal{J}) \to \mathcal{M}_{g^s,m}
    \]
    recording the smooth genus $g^s$ curve corresponding to the vertex $0$, with $m$ markings corresponding to $J$ and $K$. The map $\beta$ serves as a replacement for the map $\tau_{\cG}$ and will be used in the proof of \autoref{prop: base case I}.

\end{remark}

\begin{remark}\label{lemma: adges after spine contraction}
    Consider a contracted main component simplification described in case (c) of \autoref{def: simplification process}. If $\cG$ satisfies the fixed-domain constraint from  \autoref{def: condition *}, then the new graph $\cG^s$ satisfies
    \[
        |E(\G^s)| \geq n + 2\ell -3.
    \]
    Indeed, by \autoref{def: condition *}, the vertex $0$ is adjacent to $\ell$ other vertices of $\alpha_1,\ldots,\alpha_\ell$ satisfying $\mathrm{val}(\alpha_i) = 3$, which contributes at least $2\ell$ edges. Contracting vertex $0$ with $n-\ell$ marked points and $\ell$ edges connected to it and applying the stabilization process adds at least $n -3$ edges.
\end{remark}

\section{Proof of \autoref{thm: main II}}\label{sec: proof of main II}

\subsection{Setting the induction} 
We will prove \autoref{thm: main II} by induction on the set of $\cG$ such that
\begin{enumerate}[(i)]
    \item $\cG$ is an $(n+\ell)$-marked genus $g$ augmented graph, as in \autoref{def: augmented graphs},
    \item $\cG$ satisfies the fixed-domain constraint from \autoref{def: condition *}, 
    \item $\cG$ has weighted homology class satisfying \eqref{eqn: num constraint inequality} and \eqref{eqn: bound on k}. 
\end{enumerate}
In particular, every such $\cG$ satisfies \eqref{eqn: num constraint inequality}. The set of such augmented graphs is partially ordered as follows. Let $\widetilde{\mathsf{\Gamma}} = (\mathsf{\Gamma}, \underline{\mathsf{m}}, \underline{d}, \underline{A}, h)$ and $\widetilde{\mathsf{\Gamma}'} = (\mathsf{\Gamma}',\underline{\mathsf{m}}', \underline{d}', \underline{A}', h')$. Recall that the vertices of $\G$ and $\G'$ are ordered, so that $\bd$ and $\bd'$ are vectors of nonnegative integers of length $|V(\G)|$ and $|V(\cG')|$ respectively. We declare $\cG < \cG'$ if and only if  $\mathrm{length}(\underline{d}) < \mathrm{length}(\underline{d}')$, or  
$\mathrm{length}(\underline{d}) = \mathrm{length}(\underline{d}')$ and, at the smallest index $i$ where $d_i$ and $d'_i$ differ, we have $d_i < d'_i$.

The following notation will be useful in estimating the dimension of various moduli spaces.

\begin{notation}
    For any manifold $M$, not necessarily complex, we will write
    \[
        \dim_\C M \coloneq \frac{1}{2} \dim_\R M.
    \]
    For any Fredholm operator $L$, not necessarily complex linear, we will write
    \[
        \mathrm{ind}_\C(L) \coloneq \frac{1}{2} \mathrm{ind}_\R(L),
    \]
    and similarly for Fredholm maps between real Banach manifolds. All dimensions and indices considered in this paper are even, so these numbers are integers. This convention will make various formulae in the upcoming discussion simpler by getting rid of the factor of two. 
\end{notation}  

\begin{notation}
Let $F$ and $G$ be functions depending on $g$, $r$, and an $(n+\ell)$-marked genus $g$ augmented graph $\cG$ as above. (In applications, $F$ will be the dimension of some moduli space or the index of a Fredholm operator.) We will write $F \lesssim G$ if there is a function $H=H(r,g)$ of $g$ and $r$ which does not depend on $\cG$ (or $n$, $k$, $\ell$) such that 
\[
     F(g,r,\cG) \leq G(g,r,\cG) + H(g,r) \quad\text{for all } g, r, \cG.
\]
Similarly, we will write $F \gtrsim G$ if there exists $H=H(g,r)$  such that $F(g,r,\cG) \geq G(g,r,\cG) + H(g,r)$ for all $g, r, \cG$.
\end{notation}

\subsection{Base cases}\label{sec: base cases of the induction}

In \autoref{sec: the inductive step} we will  discuss induction with respect to the vector of degrees $\bd$ of $\cG$. The base case of the induction is when $d_\alpha \in \{0,1\}$ for all $\alpha \in V(\G) \setminus \{ 0 \}$, which corresponds to stable maps which are either constant or simple on components different from the main component $\alpha=0$. The base case is divided into two subcases: when $d_0 > 0$ (treated in \autoref{prop: base case I}) and when $d_0 = 0$ (treated in \autoref{prop: Base Case II}). We may moreover assume that the function $h$, which encodes which components have the same image, is particularly simple.

\begin{proposition}[Non-contracted main component]\label{prop: base case I}
    Suppose that $\cG$ satisfies the three conditions listed at the beginning of \autoref{sec: proof of main II}, as well as
    \begin{enumerate}[(i)]
        \item $d_{0}>0$,
        \item $d_{\alpha}\in \{0,1\}$ for all $\alpha \neq 0$,
        \item either $h^{-1}(1)= \emptyset$ or $h^{-1}(1)= \{ (0, \alpha)\}$.
    \end{enumerate}
    Then there exists $c=c(r,g)$ such for $c_1([\cG,\bm]) \geq c$ the image of the map $\tau_{\widetilde{\mathsf{\Gamma}}}$ in  \eqref{eqn: map tau tilde} has complex codimension strictly bigger than $k$ unless $d_{0}=1$, $m_\alpha=1$ for all $\alpha$ such that $d_\alpha>0$ and $|E(\G)|=0$.
\end{proposition}

The proof is rather long and will make use of the following fundamental lemma.

\begin{lemma}\label{lemma: Hurwitz}
     Let $[u,J] \in \M_{\cG}(X,\cJ)$ be as in \autoref{prop: base case I}. Let $m$, $q_i$ and $\beta$ be as in \autoref{rmk: simplication cover of spine}. The following exist:
        \begin{enumerate}[(i)]
            \item a subset $\M \subset \M_{\cG}(X,\cJ)$ containing $[u,J]$;
            \item a finite-dimensional manifold $\cH$;
            \item a map $\cH \to \M_{g,n}$; 
            \item a commutative diagram of continuous maps, with $\beta$ and $\gamma$ smooth,
            \[
            \scalebox{0.85}{$
            \xymatrix{
            & \mathcal{M} \ar[dl] \ar[dr] & \\
            \cH \ar[dr]_\gamma & & \mathcal{M}_{\cG^s}(X, \cJ) \ar[dl]^\beta \\
            & \mathcal{M}_{g^s,m } &
            }
            $}
            \]
        \end{enumerate}
        such that
        \begin{enumerate}[(i)]
            \item the map $\tau_{\cG} \colon \M \to \M_{g,n} \times X^n \times \cJ$ factors through $ \cF \hookrightarrow \cH \times \M_{\cG^s}(X, \cJ) \to \M_{g,n} \times X^n \times \cJ$, where $\cF$ denotes the fiber product of the above diagram;
            \item at every point of $\cH$,
            \begin{equation}
                \label{eqn: rank estimate}
                \dim \cH - \mathrm{rank}(d\gamma) \lesssim 4d_0;
            \end{equation}
            \item $\M_{\cG}(X,\cJ)$ can be covered by countably many subsets $\M$ as above.
        \end{enumerate}
\end{lemma}

\begin{proof}
    Write $u_0 \colon C_0 \to X$ as a composition 
    \[ 
    u_{0} \colon C_{0} \xrightarrow{\phi} C_{0}^s \to X, 
    \]
    where the second map is simple. Let $\cH$ be the Hurwitz space parametrizing data $(S,\bp,S^s,\bp^s, \varphi)$ where
        \begin{enumerate}[(i)]
            \item $(S,\bp)$ is an $n$-marked smooth complex curve of genus $g$,
            \item $(S^s, \bp^s)$ is an $m$-marked smooth complex curve of genus $g^s$,
            \item $\varphi \colon S \to S^s$ is a degree $d_0$ holomorphic map such that $\bp^s = \varphi(\bp)$  as sets (note that they have different cardinality) and $\varphi$ has the same ramification profiles as $\phi \colon C_0 \to C_0^s$, 
        \end{enumerate}
        where by the ramification profile we mean the number of ramification points $b$ and the ramification data of $\varphi$ over each of these points. The space $\cH$ is defined in such a way that for $\bp = (q_1,\ldots, q_\ell, p_{\ell+1},\ldots, p_n)$ and $\bp^s = \phi(\bp)$, the collection $(C_0, \bp, C_0^s, \bp^s, \phi)$ is an element of $\cH$. Up to discrete ambiguity, the data $(S,\bp,\varphi)$ is determined by $S'$ and a choice of $b+m$ points in $S'$. Therefore, $\cH$ is a complex manifold of dimension
        \[
            \dim_\C \cH = 3g^s-3 + m + b.
        \]
        The map $\cH \to \M_{g^s,m}$ is given by $(S,\bp,S^s,\bp^s, \varphi) \mapsto (S^s,\bp^s)$. The map $\cH \to \M_{g,n}$ is given by $(S,\bp,S^s,\bp^s,\varphi) \mapsto (S,\bp)$. 
        
        Denote by $\M \subset \M_{\cG}(X,\cJ)$ the set of all pairs $[u',J']$ such that $u' \colon C' \to X$ is close to $u \colon C \to X$ and the restriction of $u'$ to the main component factors through a covering $\varphi \colon S \to S^s$ as above. By construction, this gives us a map $\M \to \cH$ which fits into the commutative diagram above and such that $\tau_{\cG}|_{\M}$ factors through $\cF$. Moreover, $\M_{\cG}(X,\cJ)$ is second countable and stratified according to the number $m$ and the ramification profile of $\phi \colon C_0 \to C_0^s$, so  $\M_{\cG}(X,\cJ)$ can be covered by countably many subsets $\M$ as above. 

        The final part of the proof is to show that the derivative of $\gamma \colon \cH \to \M_{g^s,m}$ satisfies estimate \eqref{eqn: rank estimate}.  We claim that the rank of $(d\gamma)_h$ is at least $m-b-3$ at every point of $h \in \cH$. Indeed, let $h= (S,\bp,S^s,\bp^s,\varphi)$. The restriction of $(d\gamma)_h$ to the subspace corresponding to varying the points $\bp$,
        \begin{equation*}
            \bigoplus_{p \in \bp} T_p S \to T_h \cH
        \end{equation*}
        is given by evaluating $d\varphi$ at the points of $\bz$,
        \begin{equation*}
            \bigoplus_{p \in \bp} T_p S \xrightarrow{d\varphi} \bigoplus_{p^s \in \bp^s} T_{p^s} S^s \to T_{S^s,\bp^s} \M_{g^s,m}.  
        \end{equation*}
         This map is almost injective when restricted to the subspace corresponding to $\{ p \in \bp \ | \ d\varphi(p) \neq 0 \}$, precisely, the kernel of this restriction is at most three dimensional. Since $\varphi$ has $b$ ramification points, the size of this set is at least $\max\{ 0, m-b\}$ and we conclude that $\mathrm{rank}(d\gamma) \geq \max\{0,m-b-3\}$. Therefore,
        \[
            \dim \cH - \mathrm{rank}(d\gamma) \leq 3g^s-3 + m+b - \max\{0,m-b-3\} \lesssim 2b.
        \]
        The number of ramification points can be estimated using the Riemann--Hurwitz formula: 
        \[
            b \leq   d_0(2-2g^s) -     (2-2g) \lesssim 2d_0,
        \]
        which implies \eqref{eqn: rank estimate}, completing the proof.
\end{proof}

\begin{proof}[Proof of \autoref{prop: base case I}]
    We will first discuss the proof under the assumption $h^{-1}(1) = \emptyset$; the case $h^{-1}(1)=\{(0,\alpha)\}$ is similar and will be discussed at the end of the proof. \\

    \emph{Case 1: Simple main component}. If $d_{0} = 1$, then by \autoref{def: augmented maps}, every map in $\M_{\cG}(X,\cJ)$ is simple in the sense of \autoref{def: simple}. By the transversality result for simple maps, \autoref{cor: transversality}, the map $\tau_{\cG}$ is Fredholm of index 
    \begin{align*}
        \mathrm{ind}_{\C}(\tau_{\cG}) &= 
    c_1([\cG]) + (r-3)(1 - g - h_1(\G)) + (n+\ell)  - |E(\G)| - nr - 3(g - 1)-n \\
        &\leq c_1([\cG,\bm]) + (r-3)(1 - g) - nr - 3(g - 1) \\
        &= -k
    \end{align*} 
    The  inequality is strict unless $E(\G) = 0$ and $m_0=1$.   \\

    \emph{Case 2: Multiply covered main component}. In this case (including the three subcases below) we assume that $d_{0} > 1$. For every $[u,J] \in \M_{\cG}(X,J)$, the main component $u_0 \colon C_0 \to X$ can be written as a composition 
    \[ 
    u_{0} \colon C_{0} \xrightarrow{\phi} C_{0}^s \to X, 
    \]
    where the second map is simple. (Note that the genus $g^s$ of $C^s$ is not constant as $u$ varies in $\M_{\cG}(X,\cJ)$ but the arguments below account for all possible values of $g^s \leq g$.)
    Applying the connected multiple cover simplification described in part (a) of \autoref{def: simplification process} we obtain an augmented graph $\cG^s$ and an element of $\M_{\cG^s}(X,\cJ)$. The projection of $\tau_{\cG}([u,J])$ on $X^n \times \cJ$ agrees with the image of the simplified map under  
    \begin{equation}\label{eqn: aux }
    \mathcal{M}_{\cG^s}(X, \cJ) \to X^n \times \cJ.
    \end{equation} 
    Since all maps in $\mathcal{M}_{\cG^s}(X, \cJ)$ are simple, by \autoref{thm: transversality}, $\mathcal{M}_{\cG^s}(X, \cJ)$ is a Banach manifold and \eqref{eqn: aux } is a Fredholm map of complex index 
    \begin{equation}\label{eqn: Case I inequality}
    \begin{split}
    & \sum_{\alpha \in V(\G^s)} c_1(A_\alpha) + (r-3)(1-g^s -h_1(\G^s)) + (n+\ell) - |E(\G^s)| - nr \\
    \leq & \ c_1([\cG,\bm]) - (m_{0} d_{0} - 1) c_1(A_{0}) + (r-3) + (n+\ell) - |E(\G^s)| - nr.
    \end{split}
    \end{equation}  

    If \eqref{eqn: Case I inequality} is strictly smaller than $-k$, then the image of \eqref{eqn: aux } has complex codimension strictly bigger than $k$ and so does the image of $\tau_{\cG}$. However, \eqref{eqn: Case I inequality} is not always negative for $n$ large, and we will have to study in greater detail the relationship between $\mathcal{M}_{\cG}(X, \cJ)$ and $\mathcal{M}_{\cG^s}(X, \cJ)$. 

    To do so, we make use of \autoref{lemma: Hurwitz} as follows. Let $\cF \subset \cH \times\M_{\cG^s}(X, \cJ)$ be the fiber product of $\gamma$ and $\beta$ in that lemma. The map $\tau_{\cG}|_{\M}$  factors through
        \begin{equation}
            \label{eqn: fiber product tau}
            \cF \hookrightarrow \cH \times \M_{\cG^s}(X, \cJ) \to \M_{g,n} \times X^n \times \cJ
        \end{equation}
        and by \eqref{eqn: rank estimate} and \autoref{lem: fiber product}, there exists a submanifold $\widetilde\cF \subset \cH \times\M_{\cG^s}(X, \cJ)$ containing $\cF$ and such that the projection to $\M_{\cG^s}(X, \cJ)$ is Fredholm of index $\lesssim 4d$. Therefore, since the index of the composition of Fredholm maps is the sum of indices, the composition
        \begin{equation}
            \label{eqn: fiber product diagram}
            \widetilde \cF \hookrightarrow \cH \times \M_{\cG^s}(X, \cJ) \to \M_{g,n} \times X^n \times \cJ
        \end{equation}
        contains the image of $\tau_{\cG} |_{\M}$ and is Fredholm of index
        \begin{align}\label{eqn: index after Hurw}
            \begin{split}
            \text{index of \eqref{eqn: fiber product diagram} } 
            & \lesssim \text{index of \eqref{eqn: aux }} + 4d_0 - \dim \M_{g,n} \\
            & \lesssim c_1(A_0) + \sum_{\alpha \neq 0} c_1(A_\alpha) + (n+\ell) - |E(\G^s)|-nr+4d_0-n \\
            &\lesssim -(d_0-1) c_1(A_0)- \ell -k+4 d_0
            \end{split}
        \end{align}
       where in the last equality we have used that $|E(\G^s)| \geq |E(\G)| \geq 2\ell$ from $(ii)$ in \autoref{def: condition *} and estimated $nr$ with $c_1([\G, \mathsf{m}]) +k$. If, for $c_1([\cG, \mathsf{m}])$ sufficiently large, the quantity \eqref{eqn: index after Hurw} is less than $-k - R(r,g)$ for any function $R = R(r,g)$ of $r$ and $g$, then we would be done. However, this is not always the case.

    Before proceeding with the actual proof of \autoref{prop: base case I}, we give an estimate for $c_1(A_0)$. If $[u,J]$ is an element of $\mathcal{M}_{\cG^s}(X, \cJ)$, then $u_0 = u|_{C_0}$ is an element of the moduli space of simple maps $\M_{g^s,m}^*(X,A_0;\cJ)$ where $m$ is the number of different points in $C^s$ between $\phi(p_{\ell+1}),\ldots, \phi(p_n)$. Note that the image of \eqref{eqn: aux } is contained is some diagonal of $X^{n}$ of (complex) codimension $r(n-\ell-m)$ in $X^{n}$. The map $\M_{g^s,m}^*(X,A_0;\cJ) \to X^{m} \times \cJ$ has index
    \[
         c_1(A_0) + (r-3)(1-g^s) - (r-1)m.
    \]
    If this index is strictly smaller than $-k+r(n-\ell-m)$,  then the image of \eqref{eqn: aux } has codimension $>k$. Therefore, we may assume that
    \begin{equation}
        \label{eqn: degree lower bound}
        c_1(A_{0}) \geq (r-1)(n-\ell) - (r-3)(1-g^s)-k.
    \end{equation}
    
    We now address the case of a multiply covered main component by distinguishing three subcases. In each of them, we show that either the quantity in~\eqref{eqn: index after Hurw} or that in~\eqref{eqn: Case I inequality} is, for sufficiently large $c_1([\cG,\mathsf{m}])$, less than~$-k - R$ for any function $R = R(r,g)$ depending on $r$ and $g$.

    It will be helpful to fix $\delta,\varepsilon \in (0,1)$ such that $\varepsilon +\delta>1$ and $4 \delta < \varepsilon$. In particular, it must be $\varepsilon >4/5$. For example, $\varepsilon = 0.9$ and $\delta = 0.2$. \\
    
    \emph{Case 2a: Many markings on the main component}. 
   
        Suppose that $\ell \leq \varepsilon n$. In this case, we observe that \eqref{eqn: Case I inequality} is bounded above by
    \[
         c_1([\cG,\bm]) - c_1(A_{0}) + (r-3) + (n+\ell) - |E(\G^s)| - nr
    \] 
        Combining \eqref{eqn: degree lower bound} with $|E(\G^s)| \geq |E(\G)| \geq 2\ell$ from $(ii)$ in \autoref{def: condition *} and \eqref{eqn: num constraint inequality}, we estimate this number by
        $$
        \lesssim - (r-1)(n-\ell) + n  - \ell \leq -(r-2)(n-\ell)  \leq -(r-2)(1-\varepsilon)n 
        $$
        which would be a sufficient estimate if 
        $
        (r-2)(1-\varepsilon)> (r-1-\rho)/5.
        $
        However, this is not possible in general, as the constraints $\varepsilon + \delta > 1$ and $\varepsilon > 4\delta$ prevent it. Nevertheless, for all $\rho$, we can find such $\varepsilon,\delta$ satisfying 
        $
        (r-1)(1-\varepsilon)> (r-1-\rho)/5
        $
        and this will be enough. Indeed, from \eqref{eqn: degree lower bound} and with such choice of $\varepsilon$ and $\delta$, the quantity
        $$
        c_1(A_0) \gtrsim (r-1)(1-\varepsilon)n-k 
        $$
      diverges to infinity for large $c_1([\cG,\mathsf{m}])$ and this will allow us to conclude. We distinguish the following subcases. 
      \begin{enumerate}
        \item[$\bullet$] Suppose $\ell$ is larger than a constant $L(r,g)$ depending only on $r$ and $g$ ( and to be determined in the sequel). Since $c_1(A_0) \geq 5$, then there exists $D(r,g)$ such that if $d_0 \geq D$
        $$
        \text{Equation \eqref{eqn: index after Hurw}} \lesssim -k-d_0 \leq -k-D(r,g)
        $$
        which is as small as we need. If instead $d_0 \leq D(r,g)$, then the quantity in Equation \eqref{eqn: Case I inequality} is 
        $$
        \lesssim 4 d_0- \ell -k \leq 4D(r,g)-L(r,g) -k
        $$
        and the conclusion follows (up to increasing $L(r,g)$).
        \item[$\bullet$] Suppose instead $\ell \leq L(r,g)$. Again, there exists again a constant $D(r,g)$ such that if $d_0 \geq D(r,g)$ then we conclude by looking at  \eqref{eqn: index after Hurw}. On the other hand, if $d \leq D(r,g)$, then we have
        $$
        \text{Equation \eqref{eqn: index after Hurw}} \lesssim -c_1(A_0)-k
        $$
        and this is also enough as $c_1(A_0)$ diverges for large $c_1([\cG,\mathsf{m}])$.
        
    \end{enumerate}
        
        \emph{Case 2b: Many markings on $\P^1$ components and high degree main component}. Suppose that $\ell \geq \varepsilon n$ and $d_{0} \geq \delta n$. This case is similar, but easier. We use $|E(\G^s)| \geq |E(\G)| \geq 2\ell$, $m_{0} \geq 1$, and $c_1(A_{0}) \geq 1$ to estimate \eqref{eqn: Case I inequality} by
        \begin{align*}
           \lesssim - (d_{0} - 1) + n - \ell -k \leq n(1 - \varepsilon - \delta) + 1 -k,
        \end{align*}
        which gives the desired bound by our choice of $\varepsilon$ and $\delta$. \\

        \emph{Case 2c: Many markings on $\P^1$ components and low degree main component}. Suppose that $\ell \geq \varepsilon n$ and $d_{0} \leq \delta n$. 
    Then the quantity \eqref{eqn: index after Hurw} can be estimated as
    \[
    \lesssim 4d_0 - \ell - k \leq (4\delta - \varepsilon)n - k,
    \]
    where we have used $\ell \geq \varepsilon n$ and $d \leq \delta n$.  
    Since $\varepsilon > 4\delta$ by assumption, the proof is complete.

    \emph{Case 3: Disconnected cover involving the main component}. Suppose that $h^{-1}(1)=\{(0,\alpha)\}$. The proof is similar to the cases discussed earlier except we have to apply additionally the disconnected multiple cover simplification process described in part (b) of \autoref{def: simplification process}. Let $[u,J]$ be an element of $\M_{\cG}(X,\cJ)$.  

    Consider first the case $d_0 = 1$. Note that this is only possible if $g=0$. Denote the resulting of the simplification by $u^s:C^s \to X$ and the augmented graph modelling it by $\cG^s$. 
    We have a well-defined map  
    \[
    \rho_{\cG^s} \colon \M_{\cG^s}(X, \cJ) \to X^n \times \M_{0,n} \times \cJ,
    \]  
    given by evaluation at the points $p_1', \ldots, p_\ell', p_{\ell+1}, \ldots, p_n \in C^s$, together with the smooth curve $C_0^s = C_0$ marked at the points $C_0 \cap C_{\alpha_1}, \ldots, C_0 \cap C_{\alpha_\ell}, p_{\ell+1}, \ldots, p_n \in C_0$, and $J \in \cJ$. 

    We use the notation $\rho_{\cG}$ to distinguish this map from the standard map $\tau_{\cG}$. The two are very similar, but $\tau_{\cG}$ is technically only defined when $\cG$ satisfies the fixed-domain constraint (note that in this setting, condition (ii) of \autoref{def: condition *} may fail).

    The map $\rho_{\cG}$ satisfies $\rho_{\cG^s}([u^s,J]) = \tau_{\cG}([u,J])$. Moreover, maps in $\M_{\cG^s}(X,\cJ)$ are simple and represent homology class $[\cG^s]$ satisfying $c_1([\cG^s]) < c_1([\cG])$. Therefore, the argument used in Case 1 proves the theorem in this case.
    
    The second case is $d_0 > 1$. As in Case 2, apply the connected cover simplification process described in (a) of \autoref{def: simplification process}. The resulting augmented graph $\cG^s$ still satisfies $h^{-1}(1)=\{(0,\alpha) \}$. Therefore, we may apply the disconnected cover simplification process to $u^s$ to obtain a new map $u^{ss}$ modelled on an augmented graph $\cG^{ss}$. 
    We still have a well-defined map
    \[
        \beta \colon \M_{\cG^{ss}}(X,\cJ) \to \M_{g^{ss},m}
    \]
    as in \autoref{rmk: simplication cover of spine} and Case 2 discussed above. Moreover, $[\cG^{ss}, \bm^{ss}] = [\cG,\bm]$ and $c_1([\cG^{ss}]) < c_1([\cG])$. We can now repeat the argument used in Case 2. 
\end{proof}

\begin{proposition}[Contracted main component] \label{prop: Base Case II}
    Suppose that $\cG$ satisfies the three conditions listed at the beginning of \autoref{sec: proof of main II}, as well as
    \begin{enumerate}[(i)]
        \item $d_{0}=0$;
        \item $d_\alpha \in \{0,1\}$ for all $\alpha\in V(\G)$;
        \item $h^{-1}(1)= \emptyset$.
    \end{enumerate}
    Then there exists $C=C(r,g)$ such for $n>C$ the image of the map $\tau_{\widetilde{\mathsf{\Gamma}}}$ in  \eqref{eqn: map tau tilde} has complex codimension strictly bigger than $k$.
\end{proposition}

\begin{proof}
    Applying the contracted main component simplification in part $(c)$ of \autoref{def: simplification process} to each $[u,J] \in \M_{\cG}(X,\cJ)$ we obtain a simple map $u^s$ modelled on an augmented graph $\cG^s$.  Since all maps in $\M_{\cG^s}(X,\cJ)$ are simple, by \autoref{thm: transversality} and \autoref{lem: fredholm projection}, $\M_{\cG^s}(X,\cJ)$ is a Banach manifold and $\M_{\cG^s}(X,\cJ) \to X^n \times \cJ$ is Fredholm of complex index
    \begin{equation}\label{eqn: contracted spine}
        (r-3)(1-h_1(\G^s)) + n + \ell + c_1([\cG^s])- |E(\G^s)| -nr \leq (r-3)+n + \ell + c_1([\cG])- |E(\G^s)|-nr
    \end{equation}
    where $|E(\G^s)|$ denotes the number of edges of $\cG^s$. We will show that the right-hand side of \eqref{eqn: contracted spine} is strictly smaller than $-k$ for large $n$. The proposition will then follow. 

    In showing this, we may assume that $\ell$ is not too small. Indeed, for all $u \in \M_{\cG}(X,\cJ)$ we have $u(p_{\ell +1})= \ldots =u(p_n)$ and thus $\mathrm{ev}_{\cG}$ factors through the union of diagonals $X^{\ell+1} \subseteq X^n$ of codimension $r(n -\ell-1)$. If this codimension is greater than $k$, the statement follows. Therefore, we may assume that
    \[
        r(n-\ell-1) \leq k \quad\iff\quad  n-k/r-1 \leq \ell.
    \]
    By \autoref{lemma: adges after spine contraction}, 
    \[
        |E(\G^s)| \geq n + 2\ell - 3 \geq \ell + 2n - k/r - 4.
    \]  
    By combining this with \eqref{eqn: num constraint inequality}, we estimate the complex index \eqref{eqn: contracted spine} by
    \[
        \lesssim  c_1([\cG,\bm]) -nr -n + k/r \lesssim -k - n + k/r \lesssim -k -  c_1([\cG,\bm])/r.
    \]
    It follows that if $ c_1([\cG,\bm])$ is sufficiently large, then the complex index \eqref{eqn: contracted spine} is strictly smaller than $-k$, which completes the proof. 

\end{proof}

This concludes the discussion of the base cases.

\subsection{The inductive step}\label{sec: the inductive step}

   With the base cases established, we prove \autoref{thm: main II} by induction on the set of augmented graphs $\cG$ satisfying the three conditions listed at the beginning of \autoref{sec: proof of main II}, with respect to the order explained in the same place. (In fact, throughout the induction, the weighted homology class $[\cG,\bm]$ remains fixed but we will not use this.)

    The base case, addressed in \autoref{prop: base case I}, is when $d_\alpha \in \{0,1\}$ for all $\alpha \in V(\G)$ and $h^{-1}(1) = \emptyset$. In that case, all maps in $\M_{\cG}(X,\cJ)$ are simple. Suppose now that $\cG$ is not of that form. Then one of the following situations must occur:
    \begin{enumerate}[(i)]
        \item There exists a vertex $\alpha \neq 0$ with $d_\alpha>1$. Then, the connected cover simplification described in part (a) of \autoref{def: simplification process} yields a new augmented graph $\cG^s$ such that $\cG^s < \cG$, $\cG^s$ satisfies the fixed-domain constraint, $[\cG^s,\bm^s] = [\cG,\bm]$, and $\im(\tau_{\cG}) = \im(\tau_{\cG^s})$, and we are done by induction.
        \item There exist distinct vertices $\alpha \neq \beta$, both different from $0$, such that $h(\alpha, \beta) = 1$ and $d_{\alpha}= d_{\beta}=1$. In this case, the disconnected cover simplification described in part (b) of \autoref{def: simplification process} yields a new augmented graph $\cG^s$ such that $\cG^s < \cG$, $\cG^s$ satisfies the fixed-domain constraint, $[\cG^s,\bm^s] = [\cG,\bm]$, and $\im(\tau_{\cG}) = \im(\tau_{\cG^s})$, and we are done by induction.
        \item The case where $d_{0} = 0$, $h^{-1}(1) = \emptyset$, and $d_\alpha \in \{0,1\}$ for all vertices $\alpha$ is treated in Proposition~\autoref{prop: Base Case II}.
        \item The case where $d_0>0$, $h^{-1}(1)= \{(0,\alpha)\}$ and $d_\alpha \in \{0,1\}$ for all $\alpha \neq 0$ is treated in \autoref{prop: base case I}.
        \item The case $h^{-1}(1)=\emptyset$ and $d_{\alpha} \in \{0,1\}$ for all $\alpha \neq 0$ and $d_{0}>1$ is also treated in \autoref{prop: base case I}.
    \end{enumerate}
    This concludes the proof of \autoref{thm: main II}.

\section{\autoref{thm: main II} implies  \autoref{thm: main}}
\label{subsec: thm II to I}

The first observation is that we can restrict ourselves to graphs $\G$ of a very specific shape, described below. Namely, for $m \geq 0$ and $0 \leq \ell \leq n$, let $\mathsf{\Gamma}$ be the dual graph of a curve $C$ of the kind depicted in \autoref{figure-1}. That is: $C$ consists of a main component $C_0$ of genus $g$ carrying $n-\ell$ marked points  $p_{\ell+1}, \ldots, p_n$. Attached to $C_0$ are $m$ trees $S_1,\ldots, S_m$ of $\P^1$ components. Similarly, there are $\ell$ trees $T_1,\ldots, T_\ell$ of $\P^1$ components attached to $C_0$ at points $q_1,\ldots,q_\ell$ and each of $T_i$ contains a marked point $p_i$ for $i=1,\ldots,\ell$.  The reason we can consider only such graphs is that unless $\G$ has this shape, the stabilized curve lies in $\cM_{g,n}\setminus\M_{g,n}$ and so $\cM_{\G}(C, \bp; X,\cJ)$ is empty.

\begin{figure}[h!]
	\begin{center}
		\begin{tikzpicture} [xscale=0.3,yscale=0.3]
                \draw [ultra thick, black] (0,7) to (19,7);
                \node at (15,7) {$\bullet$};
                \node at (15,8) {$p_{\ell + 1}$};
                \node at (18,7) {$\bullet$};
                \node at (18,8) {$p_n$};
                \node at (19,6) {$C_{0}$};
                 \draw [thick, dotted] (16,6) -- (17,6);
                \draw [ultra thick, black] (1,8) to (1,3);
                \draw [ultra thick, black] (0,5) to (2,2);
                \draw [ultra thick, black] (2,3) to (0,0);
                \node at (0,-1) {$S_1$};
                \draw [ultra thick, dotted] (3,4) -- (4.5,4);
                \draw [ultra thick, black] (5,8) to (6,3);
                \draw [ultra thick, black] (6,5) to (5,1);
                \node at (5,0) {$S_m$};
                \draw [ultra thick, black] (8,8) to (8,3);
                \node at (8,4) {$\bullet$};
                \node at (7.3,4) {$p_1$};
                \node at (8,2) {$T_1$};
                \draw [ultra thick, dotted] (9,4) -- (10.5,4);
                \draw [ultra thick, black] (13,8) to (11,3);
                \draw [ultra thick, black] (11,6) to (14,1);
                \draw [ultra thick, black] (14,3) to (12,0);
                \draw [ultra thick, black] (11.6,6) to (16,4);
                \node at (13,2.65) {$\bullet$};
                \node at (11.6,2.55) {$p_\ell$};
                \node at (12,-1) {$T_\ell$};
		\end{tikzpicture}
	\end{center}
	\caption{Topological type of curves considered in \autoref{thm: main}. (This is also Figure $1$ in \cite{CL2}.) \label{figure-1}\label{singular_stable_map} }
\end{figure}
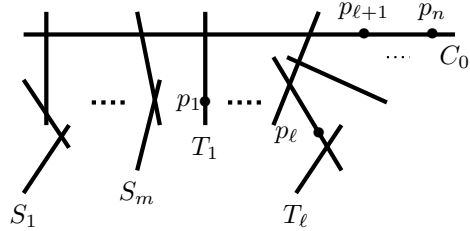

Given a $J$-holomorphic map $u \colon C \to X$ from a domain $C$ as above, its image under $\mathrm{ev}_{\G}$ is
\[
    ( u(p_1), \ldots, u(p_n), J ).
\] 
and $\pi_{\G}(u)=[C_0,q_1,\ldots, q_\ell, p_{\ell+1}, \ldots, p_n]$ where $q_i \in C_0$ are the nodes $T_i \cap C_0$ for $i=1,\ldots, \ell$. Since the trees $S_1, \dots, S_m$ do not affect the image of $\mathrm{ev}_{\mathsf{\Gamma}}$, and forgetting them strictly decreases $c_1(A)$ by the positivity assumption on $X$, we may assume $m = 0$. In fact, when $\ell = 0$ and $u_0$ is simple, the claim follows directly from \autoref{cor: transversality} together with the strict decrease of $c_1(A)$. In all other cases, the following argument applies.

Define an augmented graph $\cG$ associated with $u$ as follows:
\begin{enumerate}[(i)]
    \item relabel the markings $p_1, \ldots, p_\ell$ as $p_1', \ldots, p_\ell'$, and set $\bp' = (p_1', \ldots, p_\ell')$;
    \item for each $i = 1, \ldots, \ell$, insert a genus zero, three-valent vertex at the node $q_i$, and attach to it a new marking $p_i$, and set  $\bp = (p_1, \ldots, p_n)$;
    \item for each vertex $\alpha$ of $\G$, set $d_\alpha = m_\alpha = A_\alpha = 0$ if the restriction $u_\alpha = u|_{C_\alpha}$ is constant; otherwise, let $d_\alpha$ be the degree of $u_\alpha$ onto its image, $m_\alpha=1$, and $A_\alpha$ the homology class of the simple map underlying $u_\alpha$;
    \item for $\alpha \neq \beta$ with $d_\alpha, d_\beta > 0$, set $h(\alpha, \beta) = 1$ if and only if $\im(u_\alpha) = \im(u_{\beta})$.
\end{enumerate}

By construction, the data $\cG = (\Gamma,\bp \cup \bp',\bm,\bd,\bA,h)$ defines an augmented graph. All conditions in the fixed-domain constraint of \autoref{def: condition *} are verified directly by taking the vertex $0$ to correspond to the main component $C_0$, and observing that condition (vi) holds since we are assuming $m = 0$. The weighted homology class of $\cG$ satisfies
\[
    c_1([\cG, \mathbf{m}]) = c_1(A) = r(n + g - 1)-k.
\]
Finally, the tuple $(C, \mathbf{p} \cup \mathbf{p}', u, \mathbf{m})$ defines a weighted $(n + \ell)$-marked stable map modeled on $\cG$ 
and we have
\[
    \mathrm{ev}_{\G}( (C,\bp,u,J) ) = \mathrm{ev}_{\cG}(  (C,\bp \cup \bp',\bm,u,J) )
\]
The conclusion of \autoref{thm: main} now follows at once from \autoref{cor: main II}.

\section{Transversality}
\label{sec: transversality}

Let $\G$ be a connected $n$-pointed graph of genus $g$ as in \autoref{def: stable graph}. 
Denote by $V = V(\G)$ and $E = E(\G)$ the sets of vertices and edges of $\G$. Note that the genus of $\G$ is 
\[
    g = \sum_{\alpha \in V} g_\alpha + h_1(\G),
\]
Note that by the Euler formula,
\[
    h_1(\G) = 1 - |V| + |E|.
\]

Let $A \in H_2(X,\Z)$ be a positive class.  Consider the universal moduli space $\M_\G^*(X,\cJ; A)$  of simple pseudo-holomorphic maps $u \colon C \to X$ from $n$-marked genus $g$ domains modelled on $\G$. For the notion of a simple map, see \autoref{def: simple}

The following transversality theorem is proved when $g=0$ and $\G$ a tree in \cite[Sections 6.2, 6.3]{mcduff}, and in full generality in \cite[Section 4]{ruan1} and \cite[Section 3]{ruan2}; see also the discussion in \cite[Section 1.2]{zinger}. As \cite{ruan1,ruan2} use inhomogeneous perturbations of the Cauchy--Riemann equations, and concern a more general situation than the one in this paper, for completeness we include a short, self-contained proof.

\begin{theorem}
    \label{thm: transversality}
    $\M_\G^*(X,A,\cJ)$ is a Banach manifold and the map
    \[
        \pi \colon \M_\G^*(X,\cJ; A) \to \cJ 
    \]
    is Fredholm of complex index
    \begin{equation}
        \label{eqn:index}
        \mathrm{ind}_\C(\pi) =
        c_1(A) + (r-3)(1-g) - |E| + n.
    \end{equation}
\end{theorem}

\autoref{thm: transversality}
together with \autoref{lem: fredholm stabilization} imply the following.

\begin{corollary}
    \label{cor: transversality}
    The map
    \[
        \tau_\G \colon \M_\G^*(X,A, \cJ) \to \cM_{g,n} \times X^n \times \cJ 
    \]
    is Fredholm of complex index
    \begin{equation}
        \label{eqn:index with domain}
        \mathrm{ind}_\C(\pi) =
        c_1(A) + r(1-g) - |E| - nr.
    \end{equation}

\end{corollary}

\begin{proof}[Proof of \autoref{thm: transversality}]
    To keep the notation simple, assume that $n=0$. We will prove the theorem for $\M^*_\G(X,A,\cJ)$ replaced by the \emph{local} moduli space of simple maps from domains close to a given nodal curve $C$ modelled on $\G$, in the sense that we now explain.

    Let $\Sigma = \bigsqcup_{\alpha\in\G} C_\alpha$ be the normalization of $C$; note that $\Sigma$ is disconnected when $|V| > 1$.  We will think of $\Sigma$ as a smooth manifold equipped with an almost complex structure $j_0$. Fix also a Riemannian metric on $\Sigma$. 
    
    Let $\vv E$ be the set of pairs consisting of an edge in $E$ and a orientation on it, so that $|\vv E| = 2|E|$. We will write an element of $\vv E$ as $e = (\alpha,\beta)$ where $\alpha, \beta \in V$ are the beginning and end of the edge. The same edge with the reversed orientation will be denoted by $\bar e = (\beta,\alpha)$. 
    For every $e = (\alpha,\beta) \in \vv E$ there is a corresponding point $z_{0,e} \in C_\alpha$, such that the points $z_{0,e}$ and $z_{0, \bar e}$ map to the same node of $C$ under $\Sigma \to C$. Denote by $\bz_0$ the collection of all these points. 

    Denote by $\mathcal{P}$ the infinite-dimensional Fréchet manifold parametrizing pairs $(j,\bz)$ consisting of an (integrable) almost complex structure on $\Sigma$ and an ordered collection of $|\vv E|$ distinct points, distributed among the connected components of $\Sigma$ in the same way as the points in $\bz_0$. Set
    \[
        H = \prod_\alpha \mathrm{Diff}_+(C_\alpha),
    \] 
    where $\mathrm{Diff}_+$ denotes the group of orientation preserving diffeomorphisms. The group $H$ acts on $\mathcal{P}$, with the orbits corresponding to biholomorphism classes of marked curves of the relevant topological type. The stabilizer of $(j_0,\bz_0)$ in $H$ is the group $G =\mathrm{Aut}(\Sigma,j_0, \bz_0)$ of biholomorphisms of $(\Sigma,j_0)$ preserving every point in $\bz_0$.
    Let $\mathcal{S} \subset \mathcal{P}$ be a \emph{local Teichmüller slice} through $(j_0, \bz_0)$, characterized by the following properties:
    \begin{enumerate}[(i)]
        \item $\mathcal{S}$ is a smooth submanifold of $\mathcal{P}$ containing $(j_0,\bz_0)$ of dimension given by
        \begin{equation}
        \label{eqn:dimension of S}
        \dim_\C \mathcal{S} = 3g(\Sigma) -3|V| + |\vv E| + \dim_\C G,
     \end{equation}
    where we declare
    \[
        g(\Sigma) = \sum_{\alpha \in \G} g_\alpha;
    \]
        \item $\mathcal{S}$ is preserved by the action of $G$;
        \item the map
        \begin{equation*}
            (\mathcal{S} \times H)/ G  \to \mathcal{P}  
        \end{equation*}
        induced by the $G$-invariant multiplication map $\mathcal{S}\times H \to \mathcal{P}$, is an $H$-equivariant local homeomorphism from a neighborhood of $[j,\bz_0, \mathrm{id} ]$ to a neighborhood of $(j,\bz_0)$; in particular, the natural map
        \[
            \mathcal{S} / G \to \mathcal{P} / H
        \]
        is a local homeomorphism around $[j_0,\bz_0]$;
        \item the  tangent space to $\mathcal{S}$ at $(j_0,\bz_0)$ is transverse to the tangent space to the $H$-orbit of $(j_0,\bz_0)$ in the $L^p$ sense explained in \cite[Definition 2.49]{wendl}. 
    \end{enumerate}
    See \cite[Sections 4.2, 4.3; in particular, Theorem 4.30, Lemma 4.41, Theorem 4.43]{wendl} for a discussion of Teichmüller slices and a proof of their existence.
 
    Fix $p>2$. Let $\cB \subset W^{1,p}(\Sigma, X)$ be the space of simple maps $u \colon \Sigma \to X$ of Sobolev class $W^{1,p}$ with fixed $u_*[C_\alpha]$ for every $\alpha \in \G$; this is an open subset of $W^{1,p}(\Sigma, X)$ and so a Banach manifold. Let $\cE \to \cJ \times \cS \times \cB$ be the Banach vector bundle whose fiber over $(J, j,\bz,u)$ is the space $L^p \Omega^{0,1}(\Sigma, u^*TX)$ of $L^p$ sections of the bundle $\Lambda^{0,1} T^*\Sigma \otimes u^*TX$, where $(0,1)$ forms on $\Sigma$ are taken with respect to $j$ and the tensor product is taken over complex numbers with respect to $J$.  Consider the $G$-equivariant section
    \begin{gather*}
        \psi \colon \cJ \times \mathcal{S} \times \cB \to \cE, \\
        \psi(J,j,\bz,u) = \delbar_{J,j}(u)
    \end{gather*}
    where
    \[
        \delbar_{J,j}(u) = \frac12 (d u + J (u) \circ du \circ j)
    \]
    is the nonlinear Cauchy--Riemann operator with respect to $J$ and $j$. For a future argument, it is useful to compute the index of the restriction $\psi_J$ of $\psi$ to the slice $\{ J \} \times \cS \times \cB$.      First, the restriction of $\psi$ to every slice $\{ (J,j,\bz) \} \times \cB$ is a Fredholm section. Its index is given by the Riemann--Roch formula,
    \[
         \mathrm{ind}_\C(\delbar_{J,j}) = \sum_{\alpha\in\G} c_1(A_\alpha) + r(1-g_\alpha).
    \]
    Therefore, \autoref{lem: fredholm stabilization} and \eqref{eqn:dimension of S}, the restriction $\psi_J$ of $\psi$ to the slice $\{ J \} \times \mathcal{S} \times \cB$ is a Fredholm section of index 
    \begin{equation}
        \label{eqn:index with varying j}
        \mathrm{ind}_\C(\psi_J) =  
        \mathrm{ind}_\C(\delbar_{J,j}) + \mathrm{dim}_\C \mathcal{S} = c_1(A) + (r-3)(|\G|-g(\Sigma)) + |\vv E|) + \dim_\C G.
    \end{equation}
    
    A standard argument shows that $\psi$ is transverse to the zero section and so $\psi^{-1}(0)$ is a Banach manifold. We could then study an appropriate evaluation map defined on $\psi^{-1}(0)$, as in \cite[Section 6.3]{mcduff}. However, we choose a different approach, and instead of looking at $\psi^{-1}(0)$ we will study a map combining $\psi$ and the evaluation map. To that end, consider the $G$-invariant evaluation map
    \begin{gather*}
        \ev \colon \mathcal{S} \times \cB \to X^{\vv E} = \prod_{e\in \vv E} X  \\
        \ev(j,\bz,u) = \ev(\bz,u) = ( u(z_{e}) ).
    \end{gather*}
    Define
    \begin{gather*}
        \Psi \colon \cJ \times\mathcal{S} \times \cB \to \cE \times X^{\vv E}, \\
        \Psi(J,j,\bz,u) = ( \psi(J,j,u), \ev(\bz,u) ).
    \end{gather*}
    Consider the diagonal in $X^{\vv E}$,
    \[
        \Delta^E = \{ (x_{e}) ) \in X^{\vv E} \ | \ x_{e} = x_{\bar e} \},
    \]  
    and set 
    \[
        \M^*_{\mathrm{loc}} = \Psi^{-1}(0 \times \Delta^E)/G
    \]
    Note that $G$ acts freely on $\Psi^{-1}(0\times\Delta^E)$ by the definition of $\mathcal{B}$ as a space of simple maps. Therefore, to show that $\M^*_{\mathrm{loc}}$ is a Banach manifold, it suffices to show that $\Psi$ is transverse to $0 \times \Delta^E$, where $0 \subset \cE$ denotes the zero section. It follows then from \autoref{lem: fredholm projection} that the projection $\pi \colon \M^*_{\mathrm{loc}} \to \cJ$ is Fredholm. The index of $\pi$  is computed by subtracting from \eqref{eqn:index with varying j} the dimension of $G$ and the codimension of $\Delta^E$ in $X^{\vv E}$:
    \begin{align*}
         \mathrm{ind}_\C(\pi) &= c_1(A) + (r-3)(|\G|-g(\Sigma)) + |\vv E| - r |E| \\
        &= c_1(A) + (r-3)(1-g(\Sigma)-h_1(\G) + |E|) + (2-r)|E| \\
        &= c_1(A) + (r-3)(1-g) - |E|, 
    \end{align*}
    which agrees with \eqref{eqn:index} for $n=0$. 

    It remains to prove the transversality of the map $\Psi$ to $0 \times \Delta^E$. Let $p = (J,j,\bz,u)$ be a point in $\Psi^{-1}(0\times\Delta^E)$. Set $\bx = \ev(p)$ and denote by $N_{\bx} = T_{\bx} X^{\vv E} / T_{\bx} \Delta^E$ the normal space to $\Delta^E$ at $\bx$. We can realize $N_{\bx}$ as the subspace of $T_{\bx} X^{\vv E}$ consisting of collections  $\bv = (v_e)$ of vectors $v_e \in T_{x_e} X$ satisfying $v_e = - v_{\bar e}$ for every $e \in \vv E$. 
    Consider the operators
    \begin{gather*}
        \mathcal{L}_J \colon T_J \cJ \to L^p \Omega^{0,1}(C,j; u^*TX) \oplus N_{\bx}, \\
        \mathcal{L}_u \colon T_u \cB \to L^p \Omega^{0,1}(C,j; u^*TX) \oplus N_{\bx}
    \end{gather*}
    defined by differentiating $\Psi$ at $p$ in the direction of $J$ and $u$, respectively, and applying projection
    \[
         T_{\Psi(p)}(\cE \times X^{\vv E}) \to L^p\Omega^{0,1}(C,j; u^*TX) \oplus N_{\bx}.
    \]
    In fact, the image of $\mathcal{L}_J$ lies in the first summand as the evaluation map does not depend on $J$.
    
    We will show that $\mathcal{L} = \mathcal{L}_J + \mathcal{L}_u$ is surjective; this implies that $\Psi$ is transverse to $0\times\Delta^E$ at $p$. Since $\mathcal{L}_u$ is Fredholm (see below), the image of $\mathcal{L}$ is closed and has finite codimension. If $\mathcal{L}$ is not surjective, then, by the Hahn--Banach theorem, there exists a non-zero pair $(\eta, \bv)$ such that
    \begin{itemize}
        \item $\eta \in L^q \Omega^{0,1}(\Sigma, u^*TX)$ for $1/p + 1/q = 1$, 
        \item $\bv = (v_e) \in N_{\bx} \subset T_{\bx} X^{\vv E}$; 
        \item $(\eta,\bv)$ pairs to zero with every element in the image of $\mathcal{L}$, with respect to the pairing induced by the pairing of $L^p$ and $L^q$ and a Riemannian metric on $X$. 
    \end{itemize}
    Let $\hat u \in W^{1,p}(\Sigma, u^*TX) = T_u \cB$. We have
    \[
        \mathcal{L}_u \hat u = (D_u\hat u, ( \hat u(z_e) )  )
    \]
    where $D_u$ is the linearization of the Cauchy--Riemann operator $\delbar_{Jj}$ with fixed $(J,j)$ \cite[Section 3.1]{mcduff}. Therefore,
    \begin{equation}
        \label{eqn:variation of u}
        0 = \langle \mathcal{L}_u \hat u, (\eta,\bv) \rangle = \langle D_u \hat u, \eta \rangle + \sum_{e \in \vv E} \langle \hat u(z_e), v_e \rangle.
    \end{equation}
    In particular, for all $\hat u$ vanishing at $\bz$,
    \[
        \langle D_u \hat u, \eta \rangle = 0,
    \]
    that is: $\eta$ is a weak $L^q$ solution to the equation $D_u^*\eta =0$ in $\Sigma\setminus\bz$. Here $D_u^*$ is the formal adjoint of $D_u$: the first order differential operator defined by the property 
    \[
        \int_\Sigma \langle D_u \xi_1, \xi_2 \rangle \mathrm{vol}_\Sigma = \int_\Sigma \langle \xi_1, D_u^*\xi_2 \rangle \mathrm{vol}_\Sigma
    \]
    for all $\xi_1 \in \G(\Sigma,u^*TX)$, $\xi_2 \in \Omega^{0,1}(\Sigma, u^*TX)$ \cite[Section 3.1]{mcduff}. It follows from elliptic regularity \cite[Proposition 3.1.11]{mcduff} that $\eta$ is, in fact, of class $W^{1,p}_{\mathrm{loc}}$ in $\Sigma\setminus\bz$; therefore, continuous in that region. 

    Moreover, for every $\hat J \in T_J \cJ = C^k(X, \mathrm{End}(TX,J,\omega))$, we have
    \[
        \langle \mathcal{L}_J \hat J, \eta \rangle = 0
    \]
    Explicitly, as in \cite[Proof of Proposition 3.2.1]{mcduff}, 
    \begin{equation}
        \label{eqn:variation of J}
        \int_{\Sigma} \langle \hat J(u) \circ du \circ j, \eta \rangle \mathrm{vol}_{\Sigma} = 0
    \end{equation} 
    for every $\hat J$. It follows in a standard way from equation \eqref{eqn:variation of J} and the continuity of $\eta$ on $\Sigma \setminus \bz$ that $\eta$ vanishes on the set of injective points 
    \[
        R = \{ z \in \Sigma \setminus \bz \ | \ du(z) \neq 0, \ u^{-1}(u(z)) = \{ z \} \};
    \]
    see \cite[Proof of Proposition 3.2.1]{mcduff}.  Let $\G^* \subset \G$ be the set of vertices $\alpha \in \G$ with $A_\alpha \neq 0$ and let $\Sigma^* \subset \Sigma$ be the union of the corresponding connected components of $\Sigma$. Since $u$ is simple, $R$ is dense in $\Sigma^*$, so $\eta$ vanishes on $\Sigma^*$.    Let $\hat u \in W^{1,p}(\Sigma, u^*TX)$ be any  section (not necessarily vanishing at $\bz$) supported on $\Sigma^*$. Since $\eta = 0$ on $\Sigma^*$, 
    \[
        \langle \mathcal{L}\hat u, \bv \rangle = 0,
    \]
    which by \eqref{eqn:variation of u} is equivalent to
    \[
        \sum_{e \in \vv E^*} \langle \widehat u(z_e), v_e \rangle  = 0
    \]
    with $\vv E^* \subset \vv E$ denoting the set of oriented edges beginning in $\G^*$. 
    Since we can find $\hat u$ such that $\hat u(z_e) \in T_{x_e} X$ are arbitrary vectors, it follows that $v_e = 0$ for $e \in \vv E^*$. Note that this implies also $v_e = 0$ if $\bar e \in \vv E^*$ because $\bv$ is in $N_{\bx}$.

    It remains to show that $\eta = 0$ on $\Sigma\setminus\Sigma^*$ and $v_e = 0$ for $e \in \vv E\setminus \vv E^*$. By assumption, the subgraph $\G\setminus\G^*$ is a disjoint union of trees and $u$ is constant on each connected component of $\Sigma\setminus\Sigma^*$. Let $T$ be a connected component of $\G\setminus\G^*$ and $\Sigma_T \subset \Sigma\setminus\Sigma^*$ the corresponding union of  $\P^1$ components. Denote $u_T = u|_{\Sigma_T}$ and by $\mathcal{L}_{u_T}$ the corresponding operator. It is related to $\mathcal{L}_u$ as follows:
    \[
        \mathcal{L}_u = \bigoplus_{\alpha \in V(\G)} L_{u_\alpha} \quad\text{and}\quad \mathcal{L}_{u_T} = \bigoplus_{\alpha \in V(T)} L_{u_\alpha}.
    \]
    Since $u_T$ is locally constant, we can compute $\mathcal{L}_{u_T}$ directly. This is done in \autoref{lem:surjectivity for constant maps}  proved below, which shows that $\mathcal{L}_{u_T}$ is surjective, i.e. $\eta =0$ on $\Sigma_T$ and $v_e=0$ for $e \in \vv E_T$. 
\end{proof}

The proof of \autoref{lem:surjectivity for constant maps} is preceded by a lemma which describes the cokernel of $\mathcal{L}_u$  in the following general  situation. Let $u \colon (\Sigma,j) \to (X,J)$ be a $J$--holomorphic map from a compact, possibly disconnected Riemann surface $(\Sigma,j)$ and let $\bz = (z_i)_{i\in I}$ be a collection of distinct points in $\Sigma$ indexed by a finite set $I$. Set $\bx = (x_i)_{i\in I}$ where $x_i = u(z_i)$. Assume that $\bx$ belongs to a generalized diagonal in $X^I$. Such a diagonal is specified by an equivalence relation $\sim$ on $I$:
\[
    \Delta = \Delta_{\sim} = \{ (y_i) \in X^I \ | \ y_i = y_j \text{ if } i \sim j \}.
\]
Let $N_{\bx} = T_{\bx} X^I / T_{\bx} \Delta$ be the normal space to $\Delta$ at $\bx$.

\begin{lemma}
    \label{lem:cokernel of evaluation}
    In the situation described above, consider the operator
    \begin{gather*}
        \mathcal{L}_u \colon W^{1,p}(\Sigma,u^*TX) \to \Omega^{0,1}(\Sigma,u^*TX) \oplus N_{\bx}, \\
        \mathcal{L}_u \hat u = (D_u \hat u, [\ev_{\bz}(\hat u)] ),
    \end{gather*}
    where $D_u$ is the linearization of the Cauchy--Riemann operator $\delbar_{Jj}$ at $u$ and
    \[
        [\ev_{\bz}] \hat u = [ (\hat u(z_i))_{i\in I} ]
    \]
    is the projection of the evaluation map at $\bz$ on $N_{\bx}$. There is a short exact sequence
    \[
        \begin{tikzcd}
        0 \ar{r} & \coker( [\ev_{\bz} ]|_{\ker D_u} ) \ar{r} & \coker \mathcal{L}_u \ar{r} & \coker D_u \ar{r} & 0.
        \end{tikzcd}
    \]
\end{lemma}
\begin{proof}
    The statement follows from the snake lemma applied to the diagram
    \[
        \begin{tikzcd}
            0 \ar{r} & \ker D_u \ar{r} \ar{d}{ [\ev_{\bz}] } & W^{1,p}(\Sigma,u^*TX) \ar{r} \ar{d}{ \mathcal{L}_u} &  \ar{r} \im D_u \arrow[hookrightarrow]{d} \ar{r} & 0, \\
            0 \ar{r} & N_{\bx} \ar{r} & \Omega^{0,1}(\Sigma,u^*TX) \oplus N_{\bx} \ar{r} & L^p\Omega^{0,1}(\Sigma,u^*TX)  \ar{r} & 0. \qedhere
        \end{tikzcd} 
    \]
\end{proof}

We now apply \autoref{lem:cokernel of evaluation} to the special case of constant maps from trees, which appears in the proof of \autoref{thm: transversality}.

\begin{lemma}
    \label{lem:surjectivity for constant maps}
    Let $T$ be a tree and $\Sigma = \bigsqcup_{\alpha \in T} C_\alpha$ where $C_\alpha = \P^1$ and let $u \colon \Sigma \to X$ be a constant map with image $x \in X$.  Let $\bz$ be a collection of points in $\Sigma$ of the form $z_{\alpha\beta} \in C_\alpha$ and $z_{\beta\alpha} \in C_\beta$ for every edge $(\alpha,\beta) \in T$. Consider the diagonal
    \[
        \Delta = \Delta^{E(T)} = \{ (x_{e}) ) \in X^{\vv E(T)} \ | \ x_{e} = x_{\bar e} \},
    \]  
    In this case, the map $\mathcal{L}_u$ from \autoref{lem:cokernel of evaluation} is surjective. 
\end{lemma}

\begin{proof}
    Since $u$ is constant, the restriction of $D_u$ to every component $C_\alpha$ is the standard $\delbar$ operator
    \[
        \delbar \colon W^{1,p}(\P^1,\C) \otimes T_x X \to L^p\Omega^{0,1}(\P^1, \C) \otimes T_x X
    \]
    which is surjective as $(\coker \delbar)^* \cong H^0(\P^1, K_{\P^1}) \otimes T_x X = 0$. The kernel $\ker\delbar \cong T_x X$ is given by constant maps. Instead of the set $\vv E = \vv E(T)$ of edges with orientations, it is convenient to use the set of $E = E(T)$ of edges. To that end, pick an arbitrary orientation on each of the edges in $E$, so that 
    \[
        X^{\vv E} = \prod_{e\in E} (X\times X)
    \]
    where the first summand corresponds to the beginning and the second to the end of $E$; then $\Delta$ is the product of diagonals $X \subset X \times X$ over $e\in E$. Let $V_\alpha = T_x X$ be the copy of $T_x X$ for every $\alpha \in T$, and similarly $V_e = T_x X$ for $e \in E$.  By \autoref{lem:cokernel of evaluation}, the cokernel of $\mathcal{L}_u$ is isomorphic to the cokernel of the map
    \[
         b  \colon \bigoplus_{\alpha\in T} V_\alpha \to N_{\bx}
    \]
    where $N_{\bx}$ is the normal space to the sum of diagonals in $\bigoplus_{e\in E} V_e \oplus V_e$. Alternatively, the map above is identified with
    \[
        b \colon \colon \bigoplus_{\alpha\in T} V_\alpha \to \bigoplus_{e \in E} V_e
    \]
    where $b \colon V_\alpha \to V_e$ is the identity map when $\alpha$ is the beginning of $e$ and minus the identity when $\alpha$ is the end of $e$. It is clear from this description that adding one vertex and one edge to $T$ perserves surjectivity of $b$. Therefore, by induction with respect to $|T|$, the map $b$ is surjective for every tree $T$.
\end{proof}

\appendix
\section{Fredholm maps} 
Recall that a smooth map $f \colon X \to Y$ between Banach manifolds is said to be Fredholm if for every $p \in X$ the differential $d_p f \colon T_p X \to T_{f(p)} Y$ is a Fredholm operator. The index of $f$ at $p$ is defined as $\mathrm{ind}(d_p f)$. If $X$ is connected, which we will assume in this section, the index does not depend on $p$. In this section we will always consider $C^1$ maps whose derivative has closed image. Finally, all results of this section have obvious analogs for sections of Banach vector bundles, after replacing the differential by a covariant derivative. 

The following is used to compute in index computations throughout the paper. 

\begin{lemma}
    \label{lem: fredholm stabilization}
        Let $X$ and $Y$ be Banach manifolds and let $U_0$ and $U_1$ be finite dimensional manifolds. Let  $\overline f \colon X \times U_0 \to Y \times U_1$ be a $C^1$ map and denote its projection on $Y$ by $f$. If for every $u \in U_0$ the map $f(\cdot, u)$ is Fredholm, then $\overline f$ is Fredholm of index
    \begin{equation}
        \label{eqn: index of stabilization}
        \mathrm{ind}(\overline f) = \mathrm{ind}(f) + \dim(U_0) - \dim(U_1). 
    \end{equation}
\end{lemma}

\begin{proof}
    By looking at the derivatives, it suffices to consider the case when $X$ and $Y$ are Banach spaces, $U_0$, $U_1$ are finite dimensional vector spaces, and $\overline f$ is a bounded linear map of the form
    \begin{gather*}
        \overline f \colon X \oplus U_0 \to Y \oplus U_1, \\
        \overline f =
        \begin{pmatrix}
            f & * \\
            * & *
        \end{pmatrix}
    \end{gather*}
    such that $f \colon X \to Y$ is a Fredholm operator. The operator $f \oplus 0: X \oplus U_0 \to Y \oplus U_1$ satisfies
    \begin{align*}
        \ker(f \oplus 0) &= \ker (f) \oplus U_0, \\
        \coker (f\oplus 0) &\cong \coker(f) \oplus U_1. 
    \end{align*}
    Therefore, $f\oplus 0$ is Fredholm of index given by the right-hand side of \eqref{eqn: index of stabilization}. The difference $\overline f - f\oplus 0$ is a bounded operator  with finite-dimensional image; therefore, $\mathrm{ind}(\overline f) = \mathrm{ind}(f \oplus 0)$. 
\end{proof}

The next lemma allows us to relate the index of the projection from universal moduli spaces to the space of almost complex structures $\cJ$ to the index of the linearization with $J \in \cJ$ fixed. 

\begin{lemma}
    \label{lem: fredholm projection}
    Let $X$, $Y$, and $Z$ be Banach manifolds and let $S \subset Z$ be a  $C^1$ submanifold. Let $f \colon X \times Y \to Z$ be a smooth map which is transverse to $S$. Set $\widehat S = f^{-1}(S)$. Denote by $\pi \colon \widehat S \to Y$ the projection on $Y$, by $d_X f$ the restriction of $df$ to $TX$, and by $(\cdot)^N$ the projection on the normal bundle $N$ of $S$. For every $p = (x,y) \in \widehat S$ we have
    \begin{gather*}
        \ker (d\pi)_p \cong \ker ( (d_X f)_p^N \colon T_x X \to N_{f(p)} ), \\
        \coker (d\pi)_p \cong \coker ((d_X f)_p^N \colon T_x X \to N_{f(p)} ).
    \end{gather*}
\end{lemma}

\begin{proof}
    The statement follows from the snake lemma applied to the diagram
    \[
        \begin{tikzcd}
        0 \ar{r} & T_p \widehat S \ar{d}{(d\pi)_p} \ar{r} & T_p (X\times Y) \ar{d} \ar{r}{(d f)^N_p} \ar{r} & N_{f(p)} \ar{r} \ar{d} & 0 \\
        0 \ar{r} & T_y Y \ar{r} & T_y Y \ar{r} & 0 \ar{r} & 0 .
        \end{tikzcd}
    \]
    
\end{proof}

We will also use fiber products of Fredholm maps. Recall that the fiber product of maps $f \colon X \to Y$ and $g \colon Y \to Z$ is defined by
\[
    F = X \times_Z Y = \{ (x,y) \in X \times Y \ | \ f(x) = g(y) \},
\]
and it fits into the commutative diagram
\[
    \scalebox{0.85}{$
    \xymatrix{
    & F \ar[dl] \ar[dr] & \\
    X \ar[dr]_f & & Y \ar[dl]^g \\
    & Z &
    }
    $}
\]
The map $f$ is transverse to $g$ if for every $p=(x,y) \in F$ with value $z = f(x)=g(y)$ the map 
\[
    L_p = (df)_x - (dg)_y \colon T_x X \oplus T_y Y \to T_z Z
\]
is surjective. In that case, $F$ is a smooth submanifold of $X\times Y$. Denote by $\pi_X \colon F \to X$ and $\pi_Y \colon F \to Y$ the projection maps. By \autoref{lem: fredholm projection} applied to $S$ being the diagonal in $Z \times Z$, if $f$ is transverse to $g$, then
\begin{gather*}
    \ker( d\pi_X)_{p} ) \cong \ker( (dg)_y ) \quad\text{and}\quad     \coker( d\pi_X)_{p} ) \cong \coker( (dg)_y ), \\
    \ker( d\pi_Y)_{p} ) \cong \ker( (df)_x ) \quad\text{and}\quad     \coker( d\pi_Y)_{p} ) \cong \coker( (df)_x ).
\end{gather*}
In particular, if $f$ is Fredholm, then so is $\pi_Y$ and $\mathrm{ind}(\pi_Y) = \mathrm{ind}(f)$, and similarly for $g$ and $\pi_X$. In a non-transverse situation, the failure of transversality at $p\in F$ is measured by $\dim\coker(L_p)$, and the following generalization holds.

\begin{lemma}
    \label{lem: fiber product}
    In the situation described above, suppose that $f$ is a Fredholm map. For every $p \in F$ there exists a $C^1$ submanifold $S \subset X \times Y$ containing an open neighborhood of $p$ in $F$ and such that the projection $\pi_Y \colon S \to Y$ is a Fredholm map of index
    \[
        \mathrm{ind}(\pi_Y) = \mathrm{ind}(f) + \dim\coker(L_p).
    \]  
\end{lemma}
\begin{proof}
    Since this is a local statement, without loss of generality assume that $Z$ is a Banach space and $z=0$. Let $V = \im L_p$ and let $\pi_V \colon Z \to V$ be a projection on $V$  (since $L_p$ is Fredholm, there exists such a projection). By construction, the operator
    \[
        \pi_V \circ L_p \colon T_x X \oplus T_y Y \to V
    \]
    is surjective. Therefore,  $\pi_V\circ f$ is transverse to $\pi_V \circ g$ at $p$, and therefore at every point in some neighborhood $U \subset X \times Y$ of $p$. Therefore, 
    \[
        S = \{ (x,y) \in U \ | \ \pi_V( f(x) ) = \pi_V( g(y) )
    \]
    is a $C^1$ submanifold of $U$ containing $F$ and the statement follows from the transverse case applied to the maps $\pi_V\circ f$ and $\pi_V \circ g$. 
\end{proof}

\bibliographystyle{plain} 
\bibliography{bib}

\begin{thebibliography}{10}

\bibitem{ACGH}
E.~Arbarello, M.~Cornalba, and P.~Griffiths.
\newblock {\em Geometry of Algebraic Curves}.
\newblock Grundlehren der mathematischen Wissenschaften. Springer, Berlin,
  Germany, 2011 edition, September 2004.

\bibitem{BLLRST}
R.~Beheshti, B.~Lehmann, C.~Lian, E.~Riedl, J.~Starr, and S.~Tanimoto.
\newblock On the asymptotic enumerativity property for {F}ano manifolds.
\newblock {\em Forum Math. Sigma}, 12:Paper No. e112, 2024.

\bibitem{B1}
A.~Bertram.
\newblock {T}owards a {S}chubert {C}alculus for {M}aps from a {R}iemann
  {S}urface to a {G}rassmanian.
\newblock {\em Int. J. Math.}, 05:811--825, 1994.

\bibitem{bdw}
A.~Bertram, G.~Daskalopoulos, and R.~Wentworth.
\newblock Gromov invariants for holomorphic maps from {R}iemann surfaces to
  {G}rassmannians.
\newblock {\em J. Amer. Math. Soc.}, 9(2):529--571, 1996.

\bibitem{BP}
A.~Buch and R.~Pandharipande.
\newblock Tevelev degrees in {G}romov--{W}itten theory.
\newblock {\em arXiv:2112.14824}, 2021.

\bibitem{cd}
R.~Cavalieri and E.~Dawson.
\newblock Tropical {T}evelev degrees.
\newblock {\em arXiv 2407.20025}, 2024.

\bibitem{Cela}
A.~Cela.
\newblock Quantum euler class and virtual tevelev degrees of fano complete
  intersections.
\newblock {\em Ark. Mat.}, 61(2):301--322, 2023.

\bibitem{CIL}
A.~Cela and A.~Iribar~L\'opez.
\newblock Genus 0 logarithmic and tropical fixed-domain counts for {H}irzebruch
  surfaces.
\newblock {\em J. Lond. Math. Soc. (2)}, 109(4):Paper No. e12892, 28, 2024.

\bibitem{CL2}
A.~Cela and C.~Lian.
\newblock {Fixed-domain curve counts for blow-ups of projective space}.
\newblock {\em arXiv:2303.03433}, 2023.

\bibitem{CL}
A.~Cela and C.~Lian.
\newblock Generalized {T}evelev degrees of {$\Bbb P^1$}.
\newblock {\em J. Pure Appl. Algebra}, 227(7):Paper No. 107324, 30, 2023.

\bibitem{cl_hirzebruch}
A.~Cela and C.~Lian.
\newblock Curves on {H}irzebruch surfaces and semistability.
\newblock {\em Michigan Math. J.}, 2025.
\newblock To appear.

\bibitem{CPS}
A.~Cela, R.~Pandharipande, and J.~Schmitt.
\newblock Tevelev degrees and {H}urwitz moduli spaces.
\newblock {\em Math. Proc. Cambridge Philos. Soc.}, 173(3):479--510, 2022.

\bibitem{FL}
G.~Farkas and C.~Lian.
\newblock Linear series on general curves with prescribed incidence conditions.
\newblock {\em J. Inst. Math. Jussieu}, 22(6):2857--2877, 2023.

\bibitem{hirshi}
A.~Hirschi.
\newblock {Properties of Gromov-Witten invariants defined via global Kuranishi
  charts}.
\newblock {\em arXiv:2312.03625, to appear in Ann. Henri Lebesgue}.

\bibitem{ionel3}
E.~Ionel.
\newblock Genus {$1$} enumerative invariants in {$\bold P^n$} with fixed {$j$}
  invariant.
\newblock {\em Duke Math. J.}, 94(2):279--324, 1998.

\bibitem{ionel2}
E.~Ionel and T.~Parker.
\newblock The {G}romov invariants of {R}uan-{T}ian and {T}aubes.
\newblock {\em Math. Res. Lett.}, 4(4):521--532, 1997.

\bibitem{lian_pr}
C.~Lian.
\newblock Degenerations of complete collineations and geometric {T}evelev
  degrees of {$\mathbb{P}^r$}.
\newblock {\em J. Reine Angew. Math.}, 817:153–212, 2024.

\bibitem{LP}
C.~Lian and R.~Pandharipande.
\newblock Enumerativity of virtual {T}evelev degrees.
\newblock {\em Ann. Sc. Norm. Super. Pisa Cl. Sci.}, 26(1):71--89, 2025.

\bibitem{mo}
A.~Marian and D.~Oprea.
\newblock Virtual intersections on the {Q}uot scheme and {V}afa-{I}ntriligator
  formulas.
\newblock {\em Duke Math. J.}, 136(1):81--113, 2007.

\bibitem{mop}
A.~Marian, D.~Oprea, and R.~Pandharipande.
\newblock The moduli space of stable quotients.
\newblock {\em Geom. Topol.}, 15(3):1651--1706, 2011.

\bibitem{mcduff}
D.~McDuff and D.~Salamon.
\newblock {\em {J-holomorphic Curves and Symplectic Topology}}.
\newblock American Mathematical Society colloquium publications. American
  Mathematical Society, 2012.

\bibitem{ruan1}
Y.~Ruan and G.~Tian.
\newblock {A mathematical theory of quantum cohomology}.
\newblock {\em J. Differential Geom.}, 42(2):259--367, 1995.

\bibitem{ruan2}
Y.~Ruan and G.~Tian.
\newblock Higher genus symplectic invariants and sigma models coupled with
  gravity.
\newblock {\em Invent. Math.}, 130(3):455--516, 1997.

\bibitem{st}
B.~Siebert and G.~Tian.
\newblock On quantum cohomology rings of {F}ano manifolds and a formula of
  {V}afa and {I}ntriligator.
\newblock {\em Asian J. Math.}, 1(4):679--695, 1997.

\bibitem{tev}
J.~Tevelev.
\newblock Scattering amplitudes and stable curves.
\newblock {\em Geom. Topol.}, 2024.
\newblock To appear.

\bibitem{wendl}
C.~Wendl.
\newblock {Lectures on Holomorphic Curves in Symplectic and Contact Geometry}.
\newblock {\em arXiv:1011.1690}, 2014.

\bibitem{zinger2}
A.~Zinger.
\newblock Enumeration of genus-two curves with a fixed complex structure in
  {$\Bbb P^2$} and {$\Bbb P^3$}.
\newblock {\em J. Differential Geom.}, 65(3):341--467, 2003.

\bibitem{zinger3}
A.~Zinger.
\newblock Enumerative vs.\ symplectic invariants and obstruction bundles.
\newblock {\em J. Symplectic Geom.}, 2(4):445--543, 2004.

\bibitem{zinger}
A.~Zinger.
\newblock {Transversality for J-holomorphic maps: a complex-geometric
  perspective}.
\newblock {\em
  https://www.math.stonybrook.edu/~azinger/mat645-spr22/GrTrans.pdf}, 2022.

\end{thebibliography}

\end{document}